\documentclass[11pt, oneside]{amsart}

\usepackage{tabularx, hyperref}
\usepackage{amssymb} \usepackage{amsfonts} \usepackage{amsmath}
\usepackage{amsthm} \usepackage{epsfig, subfig}
\usepackage{ amscd, amsxtra, latexsym}
\usepackage[all]{xy}
\usepackage{caption}
\usepackage{enumerate}
\usepackage{color}

\newtheorem{lemma}{Lemma}[section]
\newtheorem{thm}[lemma]{Theorem}
\newtheorem{prop}[lemma]{Proposition}
\newtheorem{cor}[lemma]{Corollary}

\theoremstyle{definition}
\newtheorem{defn}[lemma]{Definition}

\newtheorem{quest}[lemma]{Question}

\newtheorem{rem}[lemma]{Remark}
\newtheorem{conv}[lemma]{Convention}
\theoremstyle{definition}

\long\def\comment#1\endcomment{}

\newcommand{\g} {\ensuremath {\gamma}}

\newcommand{\N}{\ensuremath {\mathbb{N}}}
\newcommand{\R} {\ensuremath {\mathbb{R}}}

\newcommand{\F} {\ensuremath {\mathbb{F}}}

\newcommand{\calC} {\ensuremath {\mathcal{C}}}

\newcommand{\calQ} {\ensuremath {\mathcal{Q}}}
\newcommand{\calP} {\ensuremath {\mathcal{P}}}

\newcommand{\calT} {\ensuremath {\mathcal{T}}}
\newcommand{\calH} {\ensuremath {\mathcal{H}}}

\newcommand{\calA} {\ensuremath {\mathcal{A}}}

\newcommand{\calU} {\ensuremath {\mathcal{U}}}

\definecolor{darkgreen}{cmyk}{1,0,1,.2}

\author{Alessandro Sisto}
\address{Mathematical Institute, 24-29 St Giles, Oxford OX1 3LB, United Kingdom}
\email{sisto@maths.ox.ac.uk}

\begin{document}

\title{On metric relative hyperbolicity}

\begin{abstract}
We show the equivalence of several characterizations of relative hyperbolicity for metric spaces, and obtain extra information about geodesics in a relatively hyperbolic space.
\par
We apply this to characterize hyperbolically embedded subgroups in terms of nice actions on (relatively) hyperbolic spaces. We also study the divergence of (properly) relatively hyperbolic groups, in particular showing that it is at least exponential.
\par
Our main tool is the generalization of a result proved by Bowditch for hyperbolic spaces: if a family of paths in a space satisfies a list of properties specific to geodesics in a relatively hyperbolic space then the space is relatively hyperbolic and the paths are close to geodesics.
\end{abstract}

\maketitle
\section{Introduction}
The notion of (strong) relatively hyperbolicity first appeared in \cite{Gro-hyp}, see also \cite{Bow-99-rel-hyp, Fa}, and has been studied only in the context of groups until Dru\c{t}u and Sapir \cite{DS-treegr} gave a metric characterization and discovered several properties of the geometry of Cayley graphs of relatively hyperbolic groups. Many examples of relatively hyperbolic spaces that are not Cayley graphs of relatively hyperbolic groups occur ``in nature'', as we will show in the last section of the present paper. It has to be noted that the term ``asymptotically tree-graded space'' in use in, e.g., \cite{DS-treegr} and \cite{D-relhyp}, is replaced everywhere in this paper by the term ``relatively hyperbolic space''. We use the latter term as the terminology in the group case is well-established and asymptotically tree-graded/relatively hyperbolic spaces are generalizations of (Cayley graphs of) relatively hyperbolic groups.
\par
We generalize the characterizations of relative hyperbolicity given in the context of groups in \cite{Gro-hyp,Bow-99-rel-hyp,GrMa-perfill} in terms of what we will call Bowditch space and in \cite{Fa} in terms of the coned-off graph, under a mild (necessary) hypothesis on the collection of candidate peripheral sets that is always satisfied for groups. These characterizations will be referred to as (RH1) and (RH2), and their precise statements are given in Section \ref{defn}. The characterization (RH0) is similar to one that can be found in \cite{D-relhyp} (which we will take as our main definition of relative hyperbolicity), except that it requires control of geodesic triangles only, as opposed to almost-geodesic triangles.
\begin{thm}
Let $X$ be a geodesic metric space and let $\calP$ be a collection of uniformly coarsely connected subsets of $X$. Then $X$ is hyperbolic relative to $\calP$ if and only if it satisfies (RH0), (RH1), (RH2) or (RH3).
\end{thm}
(Coarse connectedness is not required for (RH0),(RH3).) Moreover, we obtain some information regarding quasi-geodesics in $X$ in terms of geodesics in the Bowditch and coned-off spaces, see Corollaries \ref{rh1geod} and \ref{rh2geod}.
\par
The last characterization (RH3) is possibly, among those found so far, the characterization of relative hyperbolicity that is nearest to the standard definition of hyperbolicity. It relies on a \emph{relative Rips condition} which prescribes that every \emph{transient} point on a side of a geodesic triangle is close to a transient point on one of the other sides. A point on a geodesic is transient, roughly speaking, if it is not the midpoint of a long subgeodesic with endpoints in a suitable neighborhood of some peripheral set.
\par
An informal statement of our main tool, Theorem \ref{guessgeod}, is the following.
\par\medskip
{\bf Guessing Geodesics Lemma:} \emph{Let $X$ be a geodesic metric space and $\calP$ a collection of subsets of $X$. Suppose that for each pair of points $x,y$ a path $\eta(x,y)$ connecting them and a subset $trans(x,y)\subseteq \eta(x,y)$ have been assigned. Suppose that the pairs $trans(x,y)/\eta(x,y)$ behave like the pairs transient points/geodesics in a relatively hyperbolic space (e.g. they satisfy the relative Rips condition). Then $X$ is hyperbolic relative to $\calP$ and the set $trans(x,y)$ coarsely coincides with the set of transient points on a geodesic from $x$ to $y$.}
\par\medskip
There are a few variations of the Guessing Geodesics Lemma for hyperbolic spaces, the first of which appeared in \cite{Bow-cchyp}, where it is used to show hyperbolicity of curve complexes. Our formulation is closer to that given in \cite{Ha-cchyp}.
Here is an example of how we will use the Guessing Geodesics Lemma. When $\calP$ is a collection of (say connected) subsets of the metric space $X$, the Bowditch space associated to the pair $(X,\calP)$ is obtained gluing combinatorial horoballs to each $P\in\calP$ (the Bowditch space also depends on other choices, this is ignored here). Suppose that the Bowditch space for the pair $(X,\calP)$ is hyperbolic. Then for each pair of points $x,y$ in $X$ we can define a path connecting them by substituting each subpath of a geodesic $\gamma$ from $x$ to $y$ contained in a horoball with any path in the corresponding $P\in\calP$, and define $trans(x,y)=\gamma\cap X$. Due to the hyperbolicity of the Bowditch space, the relative Rips condition is easily checked (and the other properties required by the Guessing Geodesics Lemma are easily checked as well). Hence, we will easily obtain that if the Bowditch space for the pair $(X,\calP)$ is hyperbolic then $X$ is hyperbolic relative to $\calP$ (for $\calP$ coarsely connected).
\par
The last section is dedicated to applications. We show that a construction due to Bestvina, Bromberg and Fujiwara \cite{BBF} gives rise to relatively hyperbolic spaces. This is weaker than a very recent result by Hume \cite{H-percomm}, but it will suffice for our further applications. We also provide alternative characterizations of the notion of hyperbolically embedded subgroup, as defined in \cite{DGO}. Examples of hyperbolically embedded subgroups include peripheral subgroups of relatively hyperbolic groups and maximal virtually cyclic subgroups containing a pseudo-Anosov (resp. a rank one element, iwip, element acting hyperbolically) in a mapping class group (resp. a $CAT(0)$ group, $Out(F_n)$, a group acting acylindrically on a tree), and therefore it is a quite general notion. We show that, when $H$ is a subgroup of $G$, $H$ is hyperbolically embedded in $G$, as defined in \cite{DGO}, if and only if $G$ acts with certain properties on a relatively hyperbolic space, and also if and only if $G$ acts with certain properties on a hyperbolic space with $H$ acting parabolically (in both cases, the extra condition of properness for the actions gives back a definition of relative hyperbolicity). Here is a simplified version of one of the results.

 \begin{thm}
 Let $\{H_\lambda\}_{\lambda\in\Lambda}$ be a finite collection of subgroups of the group $G$ and let $X$ be a (possibly infinite) generating system for $G$. Then $\{H_\lambda\}$ is hyperbolically embedded in $(G,X)$ if and only if the Cayley graph $\Gamma=Cay(G,X)$ is hyperbolic relative to the left cosets of the $H_\lambda$'s and $d_\Gamma|_{H_\lambda}$ is quasi-isometric to a word metric.
\end{thm}

This result implies in particular that an infinite order element is weakly contracting as defined in \cite{Si-contr} if and only if it is contained in a virtually cyclic hyperbolically embedded subgroup, see Corollary \ref{weakcontr}. We will use it to show, with an additional hypothesis, that hyperbolically embedded subgroups are preserved in subgroups of the ambient group, i.e. that if $H$ is hyperbolically embedded in $G$ and $K<G$ then $H\cap K$ is hyperbolically embedded in $K$ (Corollary \ref{hypembsgp}). 
We will also study the divergence function of relatively hyperbolic groups. Roughly speaking, the divergence of a metric space measures the length of detours avoiding a specified ball as a function of the radius of the ball. It has been first studied in \cite{Gr-asinv} and \cite{Ger-div}, and in recent years in \cite{Be-asgeommcg,OOS-lacun,DR-divteich,DMS-div,BC-raagcones, BC-shortconj}. The divergence of one-ended hyperbolic groups is known to be exponential; the exponential upper bound is folklore, even though there seems to be no published reference for it. In Subsection \ref{div} we show the following.

\begin{thm}
 Suppose that the one-ended group $G$ is hyperbolic relative to the (possibly empty) collection of proper subgroups $H_1,\dots, H_k$. Then
$$e^n \preceq Div^G(n)\preceq \max\{e^n,e^nDiv^{H_i}(n)\}.$$
Moreover, if $G$ is finitely presented then $Div^G(n)\asymp e^n$.
\end{thm}

We make the observation that the divergence of a one-ended finitely presented group (e.g., a hyperbolic group) is at most exponential (Lemma \ref{atmostexp}), a fact that does not seem to be recorded in the literature.
\par
It is well-known that finite configurations of points in a hyperbolic space can be approximated by trees, and this is very helpful to reduce computations in hyperbolic spaces to computations in trees. In Subsection \ref{tg} we will show a similar result for relatively hyperbolic spaces.
\par
Finally, in Subsection \ref{comb} we give an example of how the Guessing Geodesics Lemma can be used to show a known combination result for relative hyperbolicity.
\par
Further applications to certain notions of boundary at infinity of a space will be presented in \cite{S-compmap}.

\subsection*{Acknowledgement} The author would like to thank Cornelia Dru\c{t}u, Vincent Guirardel and Denis Osin for very helpful comments and suggestions.

\section{Main definition of relative hyperbolicity}

The following definitions are taken from \cite{DS-treegr, D-relhyp} (except that of being weakly $(*)-$asymptotically tree-graded, which is relevant to (RH0)). We will use the notation $N_K(A)=\{x\in X| d(x,A)\leq K\}$ for the $K-$neighborhood of the subset $A$ of the metric space $X$.

\begin{defn}
\label{alpha1}
The collection $\calP$ of subsets of the geodesic metric space $X$ is said to satisfy $(\alpha_1)$ if for each $K$ there exists $B$ so that $diam(N_K(P)\cap N_K(Q))\leq B$ for each distinct $P,Q\in \calP$.
\end{defn}

\begin{defn}
The collection $\calP$ of subsets of the geodesic metric space $X$ is said to satisfy $(\alpha_2)$ if there exists $\epsilon\in (0,1/2)$ and $M\geq 0$ so that for each $P\in\calP$ and $x,y\in X$ with $x,y\in N_{\epsilon d(x,y)}(P)$ any geodesic from $x$ to $y$ intersects $N_M(P)$.
\end{defn}

 Continuous quasi-geodesics will be referred to as \emph{almost-geodesics}. Notice that in a geodesic metric space each quasi-geodesic is within bounded Hausdorff distance from an almost-geodesic (with related constants).

\begin{defn}
 The geodesic metric space $X$ is said to be $(*)-$asymptotically tree-graded (resp. weakly $(*)-$asymptotically tree-graded) with respect to the collection of subsets $\calP$ if the following holds. For every $C\geq 0$ there exist $\sigma,\delta$ so that every triangle with $(1,C)-$almost geodesic edges (resp. every geodesic triangle) $\gamma_1,\gamma_2,\gamma_3$ satisfies either
\par
$(C)$ there exists a ball of radius $\sigma$ intersecting all sides of the triangle, or
\par
$(P)$ there exists $P\in\calP$ with $N_\sigma (P)$ intersecting all sides of the triangle and the entrance (resp. exit) points $x_i$ (resp. $y_i$) of the sides $\gamma_i$ in (from) $N_\sigma(P)$ satisfy $d(y_i,x_{i+1})\leq \delta$.
\end{defn}
\begin{figure}[h]
\centering\includegraphics{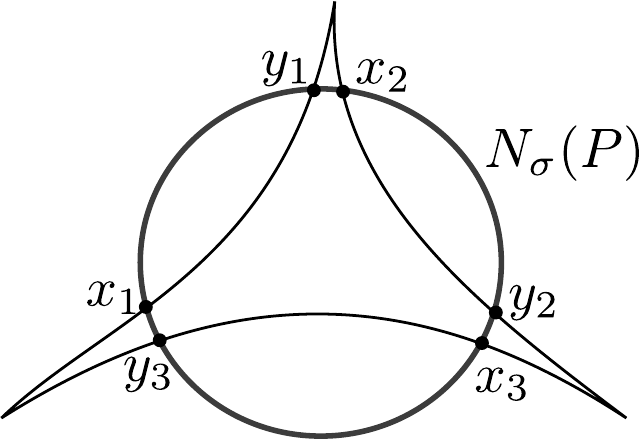}
\end{figure}

If a group is $(*)-$asymptotically tree-graded with respect to the left cosets of a collection of subgroups $H_1,\dots, H_n$ then it is said to be $(*)-$hyperbolic relative to $H_1,\dots, H_n$. This property is relevant to the Rapid Decay property, important in K-theory: it is shown in \cite{DS-RD} that if $G$ is $(*)-$hyperbolic relative to $H_1,\dots, H_n$ then $G$ satisfies the Rapid Decay property if and only if $H_1,\dots,H_n$ do.
\par
We take the following as our main definition of relative hyperbolicity.
\begin{defn}\cite{D-relhyp}
 The geodesic metric space $X$ is \emph{hyperbolic relative} to the collection of subsets $\calP$, called \emph{peripheral sets}, if $\calP$ satisfies $(\alpha_1)$ and $(\alpha_2)$ and $X$ is $(*)-$asymptotically tree-graded with respect to $\calP$.
\end{defn}

The definition above is known to be equivalent to the definition given in \cite{DS-treegr} in terms of asymptotic cones as well as the definition given in \cite{Si-proj} in terms of (almost) closest-points projections on the peripheral sets, so that we can use results from the said papers.

The following a priori less restrictive definition will turn out to be equivalent to the one above, see Proposition \ref{rh0}.
\begin{defn}
 The pair $(X,\calP)$, where $X$ is a geodesic metric space and $\calP$ is a collection of subsets, is said to satisfy (RH0) if $\calP$ satisfies $(\alpha_1)$ and $(\alpha_2)$ and $X$ is weakly $(*)-$asymptotically tree-graded with respect to $\calP$.
\end{defn}

We record the following result that will often be tacitly used.

\begin{lemma}\cite{DS-treegr}
 Let $\calP,\calQ$ be collections of subsets of the geodesic metric space $X$. Suppose that there exists $K\geq 0$ and a bijection between $\calP$ and $\calQ$ so that corresponding elements are at Hausdorff distance bounded by $K$. Then $X$ is hyperbolic relative to $\calP$ if and only if it is hyperbolic relative to $\calQ$.
\end{lemma}

We will say that a (finitely generated) group is \emph{hyperbolic relative} to the collection of its (finitely generated) subgroups $H_1,\dots,H_n$ if some (any) Cayley graph of $G$ is hyperbolic relative to the collection of the left cosets of the $H_i$'s.
\begin{conv}
 All groups we consider in the present paper will be tacitly assumed to be finitely generated.
\end{conv}

\section{The characterizations}

\label{defn}
When $\Gamma$ is a graph, we will denote by $\Gamma^0$ and $\Gamma^1$ the vertex set and edge set of $\Gamma$.
\begin{defn}

 	Suppose $\Gamma$ is a connected metric graph.
 	The \emph{horoball} $\calH(\Gamma)$ is defined to be the graph with
 	vertex set $\Gamma^0 \times \N$ and
 	edges $((v, n), (v, n+1))$ of length $1$, for all $v \in \Gamma^0$, $n \in \N$,
 	and edges $((v, n), (v', n))$ of length $e^{-n}l((v,v'))$, for all $(v, v') \in \Gamma^1$.
\end{defn}

We will call \emph{vertical ray} a geodesic ray in $\calH(\Gamma)$ obtained by concatenating all edges of type $((v,n),(v,n+1))$, for some $v\in\Gamma^0$. It is readily seen that it actually is a geodesic ray.
It can be checked that $\calH(\Gamma)$ is $\delta-$hyperbolic for some $\delta$ depending on the supremum of the lengths of the edges of $\Gamma$ (if finite). In fact, the horoballs as defined above are quasi-isometric to the combinatorial horoballs defined in \cite{GrMa-perfill}. (One can also derive this directly from a description of the geodesics in $\calH(\Gamma)$ as the concatenation of two vertical and one short horizontal segment, see \cite[Lemma 3.10]{GrMa-perfill} or \cite[Lemma 2.2]{MS-qhyp}.)
Also, for each $n$ the \emph{level} $\Gamma^0\times\{n\}$ lies within uniformly bounded Hausdorff distance from a horosphere defined in the usual sense, as a level set of a Busemann function.

\begin{defn}
 A \emph{maximal $k-$net} in a metric space $X$ is a maximal collection of distinct points at reciprocal distance at least $k$.
An \emph{approximation graph} $\Gamma$ with constants $k,R$ for a metric space $X$ is a connected metric graph whose vertex set is a maximal $k-$net in $X$ and so that two distinct vertices are connected by an edge of length $R$ if and only if their distance (as points of $X$) is at most some fixed $R$.
\end{defn}

We emphasize that approximation graphs are required to be connected, so that they can be endowed with the path metric. Such metric is clearly quasi-isometric to that of $X$ when $X$ is, say, a length space.

\begin{defn}
\label{bowsp}
 Let $X$ be a geodesic metric space and let $\calP$ be a collection of subsets of $X$. Fix approximation graphs $\Gamma_P$ for each $P\in\calP$. The \emph{Bowditch space} $Bow(X,\{\Gamma_P\}_{P\in\calP})$ is the metric space obtained identifying $\Gamma^0_P\subseteq \calH(\Gamma_P)$ with the corresponding maximal net in $P$ for each $P\in\calP$.
\end{defn}

We will often write $Bow(X)$ instead of $Bow(X,\{\Gamma_P\})$. Notice that the metric $d_X$ of $X$ and the restriction to $X$ of the metric of $Bow(X)$ are coarsely equivalent. The following is the metric analogue of the characterization of relatively hyperbolic groups that can be found in \cite{Bow-99-rel-hyp,GrMa-perfill}.

\begin{defn}
The pair $(X,\calP)$ satisfies (RH1) if each (some) Bowditch space $Bow(X)$ is hyperbolic.
\end{defn}

The equivalence of each/some, under a mild hypothesis on $\calP$, will be shown later, in Proposition \ref{rh1}.

\begin{defn}
 Let $X$ be a geodesic metric space and let $\calP$ be a collection of subsets of $X$. Fix maximal $k-$nets $N_P$ for each $P\in\calP$. The \emph{coned-off space} $\hat{X}_{\{N_P\}}$ is the metric space obtained adding segments $e_{x,y}$ of length $1$, called \emph{$P$-components}, connecting all pairs of distinct points $x,y\in N_P$, with $P\in\calP$.
\end{defn}
We will often write $\hat{X}$ instead of $\hat{X}_{\{N_P\}}$.
The paths defined below will play the role that combinatorial paths play in coned-off graphs of groups.

\begin{defn}
 We will say that two $P-$components (for the same $P\in\calP$) are \emph{tied}. A \emph{standard path} in a coned-off space $\hat{X}_{\{N_P\}}$ is a concatenation of geodesics in $X$ and $\calP$-components. A $\calP-$component of a standard path $\gamma$ is \emph{isolated} if it is not tied to any other $\calP-$components of $\gamma$. A standard path is said to be \emph{without backtracking} if all its $\calP-$components are isolated. The \emph{length} of a $\calP-$component $e_{x,y}$ is simply $d_X(x,y)$.
\end{defn}

\begin{defn}
 The coned-off space $\hat{X}_{\{N_P\}}$ is said to satisfy the BCP property if the following holds. For each $L$ there exists $K$ so that for each $L-$quasi-geodesic standard paths $\gamma_1,\gamma_2$ without backtracking and with corresponding endpoints at $d_X$-distance at most $1$:
\begin{enumerate}
 \item If $\g_1$ contains a $\calP-$component of length at least $K$ then $\g_2$ contains a $\calP-$component tied to it.
 \item If $\g_1,\g_2$ contain, respectively, the $P-$components $e_{x,y}$ and $e_{x',y'}$, for some $P\in\calP$, then $d_X(x,x'),d_X(y,y')\leq K$.
\end{enumerate}
\end{defn}

The following is the metric analogue of the characterization of relatively hyperbolic groups that can be found in \cite{Fa}.

\begin{defn}
The pair $(X,\calP)$ satisfies (RH2) if each (some) coned-off graph $\hat{X}$ is hyperbolic and the BCP property holds for $\hat{X}$.
\end{defn}

Again, the equivalence of each/some, under a mild hypothesis on $\calP$, will be shown later, in Proposition \ref{rh2}.
\par
There is (at least) one more characterization of relatively hyperbolic groups that would be interesting to turn into a metric characterization, i.e. the one in terms of relative isoperimetric inequalities given in \cite{Os-rh}.
\par
The following definition is taken from \cite{Hr-relqconv}.

\begin{defn}
\label{trans}
Let $(X,d)$ be a geodesic metric space and $\calP$ a collection of subsets. When $\alpha$ is a path in $X$ we will denote by $deep_{\mu,c}(\alpha)$ the subset of $\alpha$ defined in the following way. If $x\in\alpha$ then $x\in deep_{\mu,c}(\alpha)$ if and only if there exists a subpath $\beta$ of $\alpha$ with endpoints $x_1,x_2$ in $N_{\mu}(P)$ for some $P\in\calP$ so that $x\in\beta$ and $d(x,x_i)>c$.
\par
The set of \emph{transient} points $trans_{\mu,c}(\alpha)$ of $\alpha$ is defined as $\alpha\backslash deep_{\mu,c}(\alpha)$.
\end{defn}

\begin{rem}
 We will later show that, roughly speaking, in a relatively hyperbolic space quasi-geodesics with the same endpoints have the same transient set, up to finite Hausdorff distance (Proposition \ref{transqgeod}).
\end{rem}

Recall that the collection $\calP$ of subsets of the geodesic metric space $X$ is said to satisfy $(\alpha_1)$ if for each $K$ there exists $B$ so that $diam(N_K(P)\cap N_K(Q))\leq B$ for each distinct $P,Q\in \calP$ (Definition \ref{alpha1}).

\begin{defn}
The pair $(X,\calP)$ satisfies (RH3) if there exist $\mu,R_0$ so that
\begin{enumerate}
 \item $(\alpha_1)$ holds;
 \item for each $k$ and $R\geq R_0$ there exists $K$ so that if $d(x,P),d(y,P)\leq k$ for some $P\in\calP$ and $d(x,y)\geq K$ then $trans_{\mu,R}([x,y])\subseteq B_K(x)\cup B_K(y)$ and there exists $z\in trans_{\mu,R}([x,y])\cap N_\mu(P)$;
 \item (\textbf{Relative Rips condition}) For each $R\geq R_0$ there exists $D$ so that if $\Delta=\gamma_0\cup\gamma_1\cup\gamma_2$ is a geodesic triangle then
$$trans_{\mu,R}(\gamma_0)\subseteq N_D(trans_{\mu,R}(\g_1)\cup trans_{\mu,R}(\gamma_2)).$$
\end{enumerate}
\end{defn}

The content of (2) is that when $x,y$ are close to a peripheral set then any transient point between $x,y$ is close to either $x$ or $y$, and also there is a transient point in a specified neighborhood of the peripheral set.

In Proposition \ref{chartrans} we will show that (RH3) is another characterization of relative hyperbolicity.

\section{Guessing geodesics}

We wish to generalise the following result, which is similar to a an earlier result due to Bowditch \cite{Bow-cchyp}. For a path $\alpha$ and $p,q\in\alpha$ we denote by $\alpha|_{[p,q]}$ the subpath of $\alpha$ joining $p$ to $q$.

\begin{prop}\cite[Proposition 3.5]{Ha-cchyp}
\label{guessgeodhyp}
 Let $X$ be a geodesic metric space. Suppose that for all $x,y\in X$ there is an arc $\eta(x,y)$ connecting them so that the following hold for some $D$.
\begin{enumerate}
 \item $diam(\eta(x,y))\leq D$ whenever $d(x,y)\leq 1$;
 \item for all $x',y'\in\eta(x,y)$ we have $d_{Haus}(\eta(x,y)|_{[x',y']},\eta(x',y'))\leq D$;
 \item $\eta(x,y)\subseteq N_D(\eta(x,z)\cup \eta(z,y))$ for all $x,y,z$.
\end{enumerate}
 Then $X$ is hyperbolic and there exists $\kappa$ so that $d_{Haus}(\eta(x,y),[x,y])\leq \kappa$ for all $x,y$.
\end{prop}

The second part of the conclusion does not appear in the statement of \cite[Proposition 3.5]{Ha-cchyp}, but the proof strategy, as explained in the first paragraph of the proof, is to show the second part and then deduce the first one.

\begin{thm}
\label{guessgeod}
 Let $(X,d)$ be a geodesic metric space and $\calP$ a collection of subsets. Suppose that for each pair of points $x,y\in X$ a path $\eta(x,y)$ connecting them and a closed subset $trans(x,y)\subseteq \eta(x,y)$ have been assigned so that the following conditions hold for some large enough $D$.
\begin{enumerate}
 \item if $d(x,y)\leq 2$ then $diam(trans(x,y))\leq D$;
 \item for all $x',y'\in\eta(x,y)$, we have
$$d_{Haus}(trans(x',y'),trans(x,y)|_{[x',y']} \cup\{x',y'\})\leq D,$$
where $trans(x,y)|_{[x',y']}=trans(x,y)\cap\eta(x,y)|_{[x',y']}$;
 \item $trans(x,y)\subseteq N_D(trans(x,z)\cup trans(z,y))$ for all $x,y,z\in X$;
 \item if $x',y'\in\eta(x,y)$ do not both lie on $P\in\calP$ for any $P$, then there exists $z\in trans(x,y)$ between $x',y'$;
 \item $(\alpha_1)$;
\item for every $k$ there exists $K$ so that if $d(x,P),d(y,P)\leq k$ and $d(x,y)\geq K$ then $trans(x,y)\subseteq B_{K}(x)\cup B_{K}(y)$ and there exists $z\in trans(x,y)\cap N_D(P)$.
\end{enumerate}
Then
\begin{enumerate}[(a)]
 \item $X$ is hyperbolic relative to $\calP$;
 \item for every $C$ there exist $\mu,c_0,L$ with the following property. For any $c\geq c_0$ and $(1,C)-$almost-geodesic $\beta$ with endpoints $x,y$ the Hausdorff distance between $trans(x,y)$ and $trans_{\mu,c}(\beta)$ is at most some $L+c$.
\end{enumerate}
\end{thm}

The fact that transient sets in relatively hyperbolic spaces satisfy the properties listed above is the content of Lemma \ref{propapplies}.

We will first show part $(b)$. Given part $(b)$, we will be able to show a new characterization of relative hyperbolicity. Using this characterization, part $(a)$ follows immediately.

\begin{proof}[Proof of part $(b)$]
Let us start with some general remarks. First of all, for each $n-$gon $\eta(x_1,x_2)\cup\dots\cup\eta(x_n,x_1)$, with $n\geq 3$, we have
$$trans_{\mu,R}(\eta(x_1,x_2))\subseteq N_{(n-2)D}(trans(x_2,x_3)\cup \dots\cup trans(x_n,x_1)).$$
In particular, subdividing a geodesic from $x$ to $y$ and using $1$, we get
$$diam(trans(x,y))\leq D\max\{(d(x,y)/2-1,0\}+D\leq D d(x,y)+D.$$
\par
Notice that by $2)$ with $x'=x,y'=y$ for each $x,y$ there is $p\in trans(x,y)$ with $d(p,x)\leq D$. For convenience we will always assume $x,y\in trans(x,y)$.
\begin{lemma}
Up to increasing $D$, we can substitute $6)$ with the following:
\par\medskip
$6')$ There exists $M\geq 1$ so that for every $k\geq 1$ if $d(x,P),d(y,P)\leq k$ and $d(x,y)\geq Mk$ then $trans(x,y)\subseteq B_{Mk}(x)\cup B_{Mk}(y)$ and there exists $z\in trans(x,y)\cap N_D(P)$. 
\end{lemma}

\begin{proof}
Fix the notation of $6')$, where $M$ is large enough to carry out the following argument. Consider $x',y'\in P$ so that $d(x,x')\leq d(x,P)+1,d(y,y')\leq d(y,P)+1$. Fix $K$ as in $6)$ for $k=4D+1$. Then by $6)$ with $k=0$ we have that $trans(x',y')$ is contained in the $K-$neighborhood of $\{x',y'\}$, so that keeping into account $diam(trans(x,x'))\leq D d(x,x')+D$ and the same estimate for $y,y'$ we have
$$A=trans(x,x')\cup trans(x',y') \cup trans(y',y)\subseteq N_{Dk+K+2D}(\{x,y\}),$$
and we get $trans(x,y)\subseteq N_{Mk}(\{x,y\})$ as $trans(x,y)\subseteq N_{2D}(A)$.
The second part is obvious for $d(x,P)\leq 4D+1$ or $d(y,P)\leq 4D+1$ (as we allow an increase of $D$), so assume that this is not the case. Consider the first points $x'',y''$ in $\eta(x,x'),\eta(y,y')$ at distance $4D+1$ from $P$. By $6)$ there exists $z\in trans(x'',y'')\cap N_{D}(P)\cap B_{K}(x'')$. Such $z$ cannot be $2D-$close to $trans(x,x'')$ (keeping $2)$ into account) and $trans(y,y'')$ (as we can bound $d(x,z)$ linearly in $d(x,P)$). Hence, it is $2D-$close to $trans(x,y)$, so that $trans(x,y)$ intersects $N_{3D}(P)$. We are done up to increasing $D$ to $4D+1$.
\end{proof}

\begin{lemma}
\label{logest}
 Let $\alpha$ be any rectifiable path connecting, say, $x$ to $y$. Then $trans(x,y)\subseteq N_{f(\alpha)}(\alpha)$, where $f(\alpha)=D\log_2 \max\{l(\alpha),1\}+D$.
\end{lemma}

\begin{proof}
Indeed, this holds if $\alpha$ has length at most $2$ by condition $1)$. Suppose $l(\alpha)\geq 2$. Let $\alpha_1,\alpha_2$ be subpaths of $\alpha$ of equal length and let $z$ be the common endpoint. Then, by $3)$, $trans(x,y)\subseteq N_D(trans(x,z)\cup trans(z,y))$. On the other hand, we can assume by induction that $trans(x,z),trans(z,y)$ are both contained in the $(D\log_2 (l(\alpha)/2)+D)-$neighborhood of $\alpha$. So, $trans(x,y)$ is contained in the $(D\log_2 (l(\alpha)/2)+2D)-$neighborhood of $\alpha$, and we are done as
$$D\log_2 (l(\alpha)/2)+2D= D\log_2 l(\alpha)+D.$$
\end{proof}

We will also need ``logarithmic quasi-convexity'', i.e. the following estimate.

\begin{lemma}
\label{logqconv}
 Suppose that $\alpha$ is a $(1,C)$-almost geodesic joining $x$ to $y$. Denote
$$\xi(\alpha)=\sup\{d(p,\alpha|_{[x',y']})| x',y'\in\alpha, p\in trans(x',y')\}.$$
Then for every $P\in\calP$ we have $\alpha\subseteq N_{L}(P)$, where $L=\max\{Md(x,P)+M+3C,Md(y,P)+M+3C, \xi(\alpha)+D\}$ and $M$ is as in $6')$.
\end{lemma}

Notice that $\xi(\alpha)$ is bounded logarithmically in $l(\alpha)$ by Lemma \ref{logest}.

\begin{proof}
Set $k=\max\{d(x,P),d(y,P)\}$.
Let $\beta=\alpha|_{[x',y']}$ be so that $\beta \cap N_{k}(P)=\{x',y'\}$. Suppose by contradiction that there exists $p\in\beta$ with $d(p,P)>L$. In particular, $d(x',y')\geq 2(L-k)-3C\geq Mk$. Let $z\in trans(x',y')$ be as in $6')$. Clearly, $d(z,\beta)\geq d(p,P)-D> \xi(\alpha)$, a contradiction.
\end{proof}

We are now ready for the core of the proof and we will proceed as follows. We first show (I) that any point in $trans(x,y)$ is close to any $(1,C)$-almost-geodesic $\beta$ connecting $x$ to $y$. Then we show (II) that points in $trans_{\mu,c'}(\beta)$, for suitable $\mu,c'$, are close to $trans(x,y)$, and in order to do so we will actually show that if $p\in\beta$ is far from $trans(x,y)$ then $p\in deep_{\mu,c'}(\beta)$. Finally, we show (III) that points in $trans(x,y)$ are close to $trans_{\mu,c_0}(\beta)$ for some $c_0\geq c'$, but we will actually show that points in $deep_{\mu,c_0}(\beta)$ are far from $trans(x,y)$. Furthermore, $\mu,c'$ will be chosen so that each $deep_{\mu,c'}(\beta)$ is a disjoint union of subpaths each contained in a large subpath with endpoints in $N_\mu(P)$ for some $P\in\calP$. It is easily deduced then that $d_{Haus}(trans_{\mu,c}(\beta),trans_{\mu,c'}(\beta))\leq c-c'$ for all $c\geq c'$ (Q-C).
\par\medskip
Let now $\beta$ be a $(1,C)-$almost-geodesic connecting $x$ to $y$.
\par
{\bf Step (I).} For the purposes of this step, we can assume that $\beta$ is $C-$Lipschitz, up to increasing the constant we find with the following argument. In fact, any $(1,C)-$quasi-geodesic lies within Hausdorff distance $C'$ from a $C'-$Lipschitz $(1,C')-$quasi-geodesic, where $C'=C'(C)$. Let $p\in trans(x,y)$ be the point at maximal distance from $\beta$. We want to give a bound on $d(p,\beta)$. Let $q$ be a closest point to $p$ in $\beta$ and set $d(p,q)=\xi$. We can assume $\xi\geq 4D+1$ and $l(\beta)\geq 1$, for otherwise we are done. Also, up to substituting $\beta$ with a suitable subpath, we can assume $\xi\geq \xi(\beta)-1$, as defined in Lemma \ref{logqconv}.
\begin{figure}[h]
\centering
\includegraphics[scale=1]{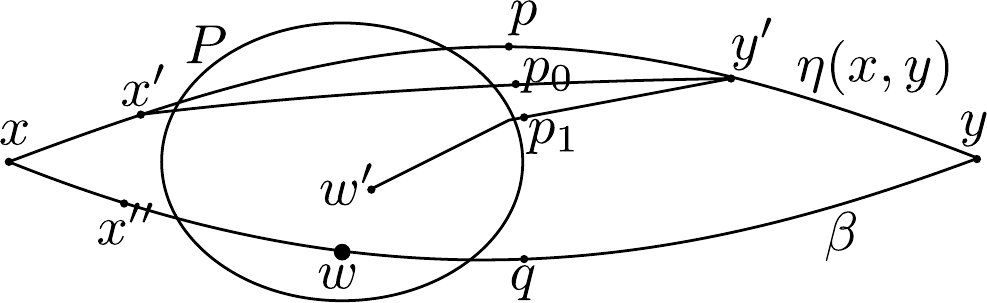}
\end{figure}

Consider the closest $x'\in trans(x,y)$ before $p$ so that $d(x',p)\geq \min\{2\xi,d(p,x)\}$, and define $y'$ similarly. By $2)$ we have $d(p,p_0)\leq D$ for some $p_0\in trans(x',y')$. The proof in the hyperbolic setting (in the case $x'\neq x, y'\neq y$) relies on $d(x',p),d(y',p)=2\xi$, but a linear upper bound in $\xi$ would make the proof work as well.
So, we would like to reduce to the case when $x'$ is either $x$ or is not too far from $p_0$ (and similarly for $y'$) by substituting $x'$ with some $w'$ if necessary. Set $w'=x'$ and $p'=p_0$ if $d(x',p)< M_0 \xi$, where $M_0=M_0(D,M,C)$ is large enough to allow us to carry out the following argument. If $d(x',p)$ is large, by $4)$ we have that $x',p$ lie in $N_{2\xi}(P)$ for some $P\in\calP$. Let $x''\in\beta$ be so that $d(x',x'')\leq \xi$. By Lemma \ref{logqconv}, $\beta|_{[x'',q]}\subseteq N_{L}(P)$, where $L\leq 3M\xi+M+D+3C$. We can choose $w\in\beta|_{[x'',q]}$, $w'\in P$ so that $d(w,w')\leq L+1$, $M+D< d(w',p_0)\leq M_0 L$. From $6')$ we have $d(p_0,trans(w',x'))>D$, so that by $3)$ (with $x',w',y')$) we have $d(p_0,trans(w',y'))\leq D$.
Fix $p_1\in trans(w',y')$ with $d(p_0,p_1)\leq D$.
Up to running the same argument on a final subpath of $\eta(w',y')$ in case $d(y',p)$ is large, we get a point $p_2$ (possibly $p_1=p_2$) with $d(p_1,p_2)\leq D$ and a point $z'$ so that $p_2\in trans(w',z')$ and $d(z',p_2),d(w',p_2)\leq M_0\xi+4D$, $d(z',\beta),d(w',\beta)\leq L+1$, but either $d(w',p_2)\geq 2\xi-4D$ or $w'=x$, and similarly for $z'$.
Consider now a path $\alpha$ obtained concatenating in the suitable order a geodesic of length at most $L+1$ from $w'$ to $a\in\beta$, a geodesic of length at most $L+1$ from $b\in\beta$ to $y'$ and $\beta|_{[a,b]}$ (choose $a=x$ and/or $b=y$ if $x'=x$ and/or $y'=y$). The length of $\alpha$ is at most $A=2L+2+(2L+2+2M_0\xi+8D+C)C$, which, upon fixing $M,D,C,M_0$, can be bounded linearly in $\xi$. Notice that $\xi=d(p,\beta)\leq d(p_2,\alpha)+4D\leq D\log_2(A)+4D$, by Lemma \ref{logest}. This gives a bound $\mu$ for $\xi$ in terms of $D,M, C$ (recall that $M_0$ depends on $D,M,C$).
\par
{\bf Step (Q-C).} We now show that $(1,C)-$almost-geodesics with endpoints in $N_{\mu}(P)$ for some $P\in\calP$ are contained in $N_{\mu'}(P)$, where $\mu'$ depends on $C$ (but not on $P$), a fact that will be needed later. As a consequence of this and $5)$, for each $c'$ large enough $deep_{\mu,c'}(\beta)$ is a disjoint union of subpaths each contained in a large subpath with endpoints in $N_\mu(P)$ for some $P\in\calP$, as subpaths of $\beta$ with endpoints in $N_\mu(P)$ and $N_{\mu}(P')$ for $P,P'\in\calP$ with $P\neq P'$ have controlled intersection. Suppose that an almost-geodesic $\alpha$ as above contains a subpath $\alpha'$ outside $N_{\mu+D+1}(P)$ with endpoints $x,y$. Then, if by contradiction $d(x,y)$ is larger than some suitable $R$, by $6)$ there exists $z\in trans(x,y)\cap N_D(P)$. But such point cannot be $\mu-$close to $\alpha$. So, $d(x,y)\leq R$ and hence $\alpha'\subseteq N_{\mu'=\mu+D+R+C+1}(P)$, which in turn gives $\alpha\subseteq N_{\mu'}(P)$.
\par
{\bf Step (II).} Fix $c'$ as determined in (Q-C). Let now $p\in\beta$ be at distance at least $\max\{c'+\mu,2\mu+C\}+1$ from $trans(x,y)$. We wish to show that $p\in deep_{\mu,c'}(\beta)$.
Let $p_1,p_2$ be the closest points on the sides of $p$ in $\beta$ at distance at most $\mu$ from $trans(x,y)$. Clearly, $d(p_1,p_2)\geq \max\{2c',2\mu+2C\}+1$.
Let $q_1$ be the last point on $trans(x,y)$ so that $d(q_1,\beta_1)\leq \mu$, where $\beta_1$ is the sub-quasi-geodesic of $\beta$ with final point $p_1$. Define $q_2$ and $\beta_2$ similarly. Notice that there are no points in $trans(x,y)$ between $q_1$ and $q_2$, as any point in $trans(x,y)$ is $\mu-$close to either a point in $\beta_1$ or a point in $\beta_2$ and $d(\beta_1,\beta_2)\geq d(p_1,p_2)-2C>2\mu$. In particular, $q_1,q_2$ both lie on some $P\in\calP$ by $4)$, which implies (considering $\beta|_{[p_1,p_2]}$) that $p$ is $(\mu,c')$-deep, as required.
\par
{\bf Step (III).} Finally, we have to show that if $p\in deep_{\mu,B}(\beta)$, where $B$ will be determined later, then $d(p,trans(x,y))> \mu$. Let $p_1,p_2$ be the endpoints of the subpath of $\beta$ contained in the neighborhood of some $P\in\calP$ as in the definition of $deep_{\mu,B}(\beta)$ (in particular, $d(p_1,p_2)\geq 2B$). Notice that $p_1,p_2\in trans_{\mu,B}(\beta)$ if $B\geq diam(N_{\mu'}(P)\cap N_{\mu'}(P'))$ for any $P,P'\in\calP$ with $P\neq P'$ and for $k'$ as in (Q-C). We have, by Step (II), $d(p_i,q_i)\leq \max\{c'+\mu,2\mu+C\}+1=A$ for some $q_i\in trans(x,y)$. In particular, $d(q_i,P)\leq A+\mu$ so that for $B$ large enough we have, by $6)$ and $2)$,
$$trans(x,y)\subseteq N_{K+D}(\eta|_{[x,q_1]}\cup\eta|_{[q_2,y]}),$$
for $K$ as in $6)$ with $k=A+\mu$.
This easily implies the claim, up to further increasing $B$, because for $B$ large enough
$$d(x,p)\notin [0,d(x,q_1)+\mu+K+D]\cup [d(x,q_2)-\mu-K-D,d(x,y)]$$
as $\beta$ is a $(1,C)-$quasi-geodesic.
\par
This concludes the proof of part $(b)$.
\end{proof}

As mentioned earlier, we now use $(b)$ to give the characterization of relative hyperbolicity (RH3) below. In the setting of Theorem \ref{guessgeod}, the conditions stated in (RH3) are readily checked in view of $(b)$, as they basically are some of the condition in Theorem \ref{guessgeod} with $trans_{\mu,R}$ substituting $trans$. (Substituting $trans$ with $trans_{\mu,R}$ allows us to drop some of the ``coherence'' conditions we had to require in Theorem \ref{guessgeod}.) In particular, the implication $\Leftarrow$ below concludes the proof of Theorem \ref{guessgeod}.

\begin{prop}[Characterization (RH3)]
\label{chartrans}
The geodesic metric space $X$ is hyperbolic relative to $\calP$ $\iff$ there exist $\mu,R_0$ so that
\begin{enumerate}
 \item $(\alpha_1)$;
 \item for each $k$ and $R\geq R_0$ there exists $K$ so that if $d(x,P)+d(y,P)\leq k$ for some $P\in\calP$ and $d(x,y)\geq K$ then $trans_{\mu,R}([x,y])\subseteq B_K(x)\cup B_K(y)$ and there exists $z\in trans_{\mu,R}([x,y])\cap N_\mu(P)$;
 \item for each $R\geq R_0$ there exists $D$ so that if $\Delta=\gamma_0\cup\gamma_1\cup\gamma_2$ is a geodesic triangle then
$$trans_{\mu,R}(\gamma_0)\subseteq N_D(trans_{\mu,R}(\g_1)\cup trans_{\mu,R}(\gamma_2)).$$
\end{enumerate}
\end{prop}

\begin{proof}
 $\Leftarrow:$ First of all, let us show that any geodesic $\gamma$ with endpoints $y_1,y_2$ in $N_\mu(P)$, for some $P\in\calP$, is contained in $N_{\mu'}(P)$ for a suitable $\mu'$. Indeed, suppose that $x\in\gamma$ lies outside $N_\mu(P)$ and let $[x_1,x_2]$ be the minimal subgeodesic of $\gamma$ containing $x$ with endpoints in $N_\mu(P)$. Consider the points $x'_i\in[x_1,x_2]$ with $0<d(x_1,x'_i)\leq 1$. If $d(x_1,x_2)$ was sufficiently large, we would get a contradiction with $2)$ as $[x'_1,x'_2]\cap N_\mu(P)=\emptyset$.
\par
As a consequence, we see that, given $R\geq R_0$ large enough, $deep_{\mu,R}(\gamma)$ is a disjoint union of subgeodesics each with both endpoints in $N_\mu(P)$ for some $P\in\calP$, for any geodesic $\gamma$. In fact, if $\gamma_1,\gamma_2$ are subgeodesics of $\gamma$ so that the endpoints of $\gamma_i$ are in $N_\mu(P_i)$ and $P_1\neq P_2$ then we can assume by $1)$ that $diam(\gamma_1\cap\gamma_2)\leq R$. Fix such $R$ from now on and let $D$ be the constant given by $3)$.
\par
Another consequence, in view of $2)$, is that there exists $E$ so that for all $w,z\in [x,y]$ we have
$$d_{Haus}(trans_{\mu,R}([w,z]), (trans_{\mu,R}([x,y])\cap[w,z]) \cup \{w,z\})\leq E.$$
(Just consider the cases when $w,z$ are/are not in $trans_{\mu,R}([x,y])$.)
\par
In particular, it is now readily checked that all conditions of Theorem \ref{guessgeod} are satisfied for $\eta(x,y)=[x,y]$ any geodesic and $trans(x,y)=trans_{\mu,R}(x,y)$ (with $\calP$ substituted by $\{N_{\mu'}(P)\}_{P\in\calP}$ in order to ensure $4)$). In view of part $(b)$, the argument to show $(*)-$asymptotic tree-gradedness below work, with suitable constants, for $(1,C)-$almost-geodesics as well as geodesics. For simplicity, we will spell out the proof for geodesics only.
\par
Let us check $(*)-$asymptotic tree-gradedness. We will use the readily checked property that for each geodesic $n-$gons $\gamma_1\cup\dots\cup\gamma_{n}$, with $n\geq 3$, we have
$$trans_{\mu,R}(\gamma_1)\subseteq N_{(n-2)D}(trans_{\mu,R}(\g_2)\cup \dots\cup trans_{\mu,R}(\g_n)).$$
Let us also show the following preliminary fact.
\par
{\bf Claim 1:} Let $\alpha_i=[y_i,z_i]$ be, for $i=0,1,2$, a geodesic with endpoints in $N_{\mu}(P_i)$ for some $P_i\in\calP$. Suppose $d(z_i,y_{i+1})\leq K$. Then if $l(\alpha)\geq Q=Q(K)$ we have $P_0=P_1=P_2$.
\par\medskip
If $Q$ is large enough we can argue as follows. By $2)$ and the quasi-convexity (which can be shown with the same argument for almost-geodesics as well) of each $N_{\mu}(P)$ there exists $L$ so that for all $w\in [y_0,z_0]$ we have $trans_{\mu,R}([w,z_0])\subseteq N_L(\{w,z_0\})$. Consider $w_0\in[y_0,z_0]$ so that
$$d(w_0,y_0)=1+\sup\{diam(N_{B}(P)\cap N_{B}(P'))|P,P'\in \calP, P\neq P'\}$$
for $B=\max\{K+\mu,L+3D+\mu'\}$. Consider a pentagon with vertices $w_0,z_0,y_1,z_1,y_2$. By $2)$ there exists $p_0\in trans([w_0,y_2])\cap N_{\mu}(P_2)$. It is readily checked that $p_0$ cannot be $3D-$close to $trans_{\mu,R}([z_0,y_1])\cup trans_{\mu,R}([y_1,z_1])\cup trans_{\mu,R}([z_1,y_2])$ (for $Q$ large enough) and can be $3D-$close to $trans_{\mu,R}([w_0,z_0])$ only if $d(w_0,p_0)\leq L+3D$. In particular, $w_0,p_0\in N_B(P_0)\cap N_B(P_2)$, which implies $P_0=P_2$ by the choice of $d(w_0,y_0)$.
\par
 Consider a geodesic triangle with sides $\gamma_0,\gamma_1,\gamma_2$ with $\gamma_i=[x_i,x_{i+1}]$. Suppose that there exists $i$ and $y\in trans_{\mu,R}(\gamma_i)$ so that $d(y,trans_{\mu,R}(\gamma_{i\pm 1})\leq Q+D$, for $Q$ as in Claim 1 with $K=D+2R$. Then the triangle clearly falls into case $(C)$. Hence, suppose that this is not the case. Then there exist deep components $[y_i,z_i]\subseteq\gamma_i$ with $d(z_i,y_{i+1})\leq D+2R$. Let $P_i$ be so that $y_i,z_i\in N_\mu(P_i)$. Claim 1 implies $P_0=P_1=P_2$, and we are done.
\par\medskip
Let us now verify the following stronger version of $(\alpha_2)$.
\par
{\bf Claim 2:} If $x_1,x_2\in X$ and $P\in\calP$ have the property that $d(x_1,P)+d(x_2,P)\leq d(x_1,x_2)-L$ then on each geodesic from $x_1$ to $x_2$ there exist $x'_1,x'_2\in N_{\mu+2D}(P)\cap trans_{\mu,R}([x_1,x_2])$ with $|d(x_i,x'_i)-d(x_i,P)|\leq L$.
\par\medskip
Let $w'_i\in P$ be so that $d(x_i,w'_i)\leq d(x_i,P)+1$ and let $w_i\in[x_i,w'_i]$ be so that $d(w_i,w'_i)=\mu+2D+2$ if such $w_i$ exists. Otherwise, a similar argument can be carried out with $w_i=x_i$ and considering a triangle instead of a quadrangle. Consider a geodesic quadrangle with vertices $x_1,x_2,w_2,w_1$. By $2)$, there exists $z_i\in trans_{\mu,R}([w_1,w_2])\cap N_\mu(P)\cap B_{K}(w_i)$, for a suitable $K$ and if $L$ is large enough. Clearly, $d(z_i,[x_i,w_i])>2D$. Also, for $L$ large enough $d(z_i,[x_{i+1},w_{i+1}])>2D$ because
$$d(x_1,x_2)\leq d(x_i,w_i)+d(w_i,z_i)+d(z_i,[x_{i+1},w_{i+1}])+d(x_{i+1},P)\leq$$
$$d(x_1,P)+K+d(z_i,[x_{i+1},w_{i+1}])+d(x_2,P).$$
Hence, $d(z_i,trans_{\mu,R}([x_1,x_2]))\leq 2D$. If $x'_i\in trans_{\mu,R}([x_1,x_2])$ is so that $d(z_i,x'_i)\leq 2D$ then it clearly satisfies the requirements, up to increasing $L$ again.
\par
$\Rightarrow:$ We will show (RH3) starting from the weaker hypothesis (RH0) in Proposition \ref{rh0}.
\end{proof}

In retrospect, i.e. once the proposition is established, we see that in the first part of the proof of the proposition we have shown the following lemma.

\begin{lemma}
\label{propapplies}
 Let $X$ be hyperbolic relative to $\calP$. All hypotheses of Theorem \ref{guessgeod} are satisfied substituting $\calP$ with $\{N_{\mu'}(P)\}_{P\in\calP}$ and for $\eta(x,y)=[x,y]$ a geodesic, $trans(x,y)=trans_{\mu,R}([x,y])$, for some suitably chosen $\mu,\mu',R$.
\end{lemma}

\begin{defn}
 A subset $P$ of a metric space $X$ is \emph{$K-$coarsely connected} if for all $x,y\in P$ there exists a chain $x=x_0,\dots,x_n=y$ of points $x_i\in P$ satisfying $d(x_i,x_{i+1})\leq K$. A collection $\calP$ of \emph{uniformly coarsely connected} subsets is a collection of $K-$coarsely connected subsets for some $K$.
\end{defn}

Recall that we defined Bowditch spaces in Definition \ref{bowsp}.

\begin{prop}[Characterization (RH1)]
\label{rh1}
 Let $X$ be a geodesic metric space and let $\calP$ be a collection of uniformly coarsely connected subsets of $X$. Then $(X,\calP)$ is relatively hyperbolic $\iff$ each (some) Bowditch space $Bow(X)$ is hyperbolic.
\par
Moreover, given a model of $Bow(X)$ there exist $K,\mu,R_0$ so that for each geodesic $\gamma$ in $Bow(X)$ connecting $x, y\in X$ we have that $\gamma\cap X$ is within Hausdorff distance $K+R$ from $trans_{\mu,R}([x,y])$ for each $R\geq R_0$.
\end{prop}
The following is an easy consequence of the proposition.

\begin{cor}
\label{rh1geod}
Suppose that $(X,\calP)$ is relatively hyperbolic and fix a model for $Bow(X)$. Then there exists $K$ so that if $\gamma$ is a geodesic in $Bow(X)$ then substituting each subpath of $\gamma$ lying outside $X$ with a geodesic in $X$ we get a $(K,K)-$quasi-geodesic.
\end{cor}

\begin{proof}[Proof of Proposition \ref{rh1}]
 $\Leftarrow:$ First of all, let us show that all combinatorial horoballs $H=\calH(P)$ in $Bow(X)$ are within bounded Hausdorff distance from actual horoballs. As this is true in the path metric of $H$, we just need to show that any geodesic with endpoints in $H$ is contained in $H$ except in uniformly bounded balls around its endpoints. This is easy to see: consider $x,y\in P$, a geodesic $\gamma$ in $Bow(X)$ connecting them and the vertical rays $\gamma_x,\gamma_y$ at $x,y$. Such vertical rays are geodesic rays in $Bow(X)$, and their Hausdorff distance is finite. The ideal triangle $\gamma,\gamma_x,\gamma_y$ is thin, and this easily implies that $\gamma$ is contained in $H$, except in balls around the endpoints of radius bounded in terms of the hyperbolicity constant.
\par
For convenience, we will assume that each $P\in\calP$ is path-connected, which we can guarantee by adding suitable paths to $X$ in view of uniform coarse connectedness. We can use Theorem \ref{guessgeod} applied to $\eta(x,y)$ constructed substituting each subpath of $[x,y]\subseteq Bow(X)$ contained in some combinatorial horoball with a path in $P$ with the same endpoints, with $trans(x,y)$ being $[x,y]\cap X$. In view of the previous paragraph and the fact that the metric $d_{Bow(X)}$ is coarsely equivalent to $d_X$ on $X$, all conditions are readily checked. More precisely, $1)$ and $4)$ are clear, $5)$ and $6)$ follow from the corresponding statements in $Bow(X)$ and $2)$ can be proven in a similar way to $3)$, which we are about to show. Consider a geodesic triangle in $Bow(X)$ with vertices $x,y,z\in X$. Consider a point $w\in[x,y]\cap X$ and assume without loss of generality that there exists $p\in[y,z]$ with $d_{Bow(X)}(p,w)\leq \delta$, the hyperbolicity constant. As $y,z\in X$ and we can regard the levels of $H$ as horospheres, there is on $[y,z]\cap X$ a point $q$ not too far away from $p$, meaning that $d_{Bow(X)}(p,q)$ is bounded by some constant depending on $Bow(X)$ only. Given this, $3)$ follows from the coarse equivalence of $d_{Bow(X)}$ and $d_X$ on $X$.
\par
$\Rightarrow$: As earlier, assume that $P$ is path-connected. Consider any Bowditch space $Bow(X)$. Let $\phi:Bow(X)\to X$ be the map restricting to the identity on $X$ and mapping in the natural way each combinatorial horoball to the corresponding $P\in\calP$. Define $\eta(x,y)$ to be the concatenation of a geodesic in the horoball from $x$ to $\phi(x)$, any geodesic $\gamma_{x,y}$ in $X$ from $\phi(x)$ to $\phi(y)$ and a geodesic in the horoball from $\phi(y)$ to $y$. Also, let $trans(x,y)=trans_{\mu,R}(\gamma_{x,y})$, where $\mu,R$ are given by Lemma \ref{propapplies}. The hypotheses of Theorem \ref{guessgeod} are satisfied in the metric of $Bow(X)$ as well, because the restriction of such metric on $X$ is coarsely equivalent to the metric of $X$. In particular, $Bow(X)$ is hyperbolic relative to the collection of combinatorial horoballs, which are themselves hyperbolic. Hence, $Bow(X)$ is hyperbolic.
\end{proof}

\begin{prop}[Characterization (RH2)]
\label{rh2}
 Let $X$ be a geodesic metric space and let $\calP$ be a collection of uniformly coarsely connected subsets of $X$. Then $(X,\calP)$ is relatively hyperbolic $\iff$ each (some) coned-off graph $\hat{X}$ is hyperbolic and the BCP property holds for $\hat{X}$.
\par
Moreover, given a model of $\hat{X}$ there exist $K, \mu, R_0$ so that geodesics in $X$ are within Hausdorff $\hat{X}-$distance $K$ from geodesics in $\hat{X}$ and, conversely, if $\gamma$ is a geodesic in $\hat{X}$ connecting $x,y\in X$ then $\gamma\cap X$ is within Hausdorff $X-$distance $K+R$ from $trans_{\mu,R}([x,y])$, for each $R\geq R_0$.
\end{prop}

\begin{cor}
\label{rh2geod}
Suppose that $(X,\calP)$ is relatively hyperbolic and fix a model for $\hat{X}$. Then there exists $K$ so that if $\gamma$ is a geodesic in $\hat{X}$ then substituting each component of $\gamma$ with a geodesic in $X$ we get a $(K,K)-$quasi-geodesic.
\end{cor}

\begin{proof}[Proof of Proposition \ref{rh2}]
 $\Rightarrow:$ Hyperbolicity of $\hat{X}$ and the fact that geodesics in $X$ are close in the $d_{\hat{X}}$ metric to $\hat{X}-$geodesics follow from Proposition \ref{guessgeodhyp} using Proposition \ref{chartrans}-$(3)$ (the requirement for the paths $\eta(x,y)$ to be arcs is inessential, as one can see taking a product of $X$ with $[0,1]$, for example). We have to show the BCP property. For each $P\in\calP$ denote by $\pi_P:X\to P$ be a map satisfying $d(x,\pi_P(x))\leq d(x,P)+1$. We will use the following lemmas.

\begin{lemma}
 Let $(X,\calP)$ be relatively hyperbolic. There exists $K_0$ with the following property. Consider a standard path $\alpha:I\to\hat{X}$ with endpoints $x,y\in X$ that does not contain a $P-$component, for some $P\in\calP$. Then
$$d(\pi_P(x),\pi_P(y))\leq K_0 l(I)+K_0.$$
\end{lemma}

\begin{proof}
$\pi_P$ is coarsely Lipschitz in $X$ and also for each $P'\in\calP$ with $P'\neq P$ we have a bound on $diam(\pi_P(P'))$ \cite{Si-proj}. Hence, in order to prove the lemma we just need to consider a subdivision of $\alpha$.
\end{proof}

\begin{lemma}
 Let $(X,\calP)$ be relatively hyperbolic. Then there exists $K_1$ with the following property. For each $L$ there exists $K_2$ so that if the $L-$quasi-geodesic $\alpha$ in $\hat{X}$ with endpoints $x,y\in X$ does not intersect $B^{\hat{X}}_{K_2}(P)$, for some $P\in\calP$, then
$$d(\pi_P(x),\pi_P(y))\leq K_1.$$
\end{lemma}

\begin{proof}
As $\hat{X}$ is hyperbolic and geodesics in $X$ are close to geodesics in $\hat{X}$, by taking $K_2$ large enough we can assume that a geodesic in $X$ from $x$ to $y$ stays as far as we wish, in $\hat{X}$ and hence in $X$, from $P$. The uniform bound then follows from $(AP'2)$ in \cite{Si-proj}.
\end{proof}

\begin{lemma}
 Let $(X,\calP)$ be relatively hyperbolic. Then for each $L$ there exists $K_3$ with the following property. Let $\alpha$ be an $L-$quasi-geodesic standard path without backtracking with endpoints $x,y$. Suppose that $\alpha$ contains a $P-$component $e_{p,q}$, for some $P\in\calP$. Then $d(p,\pi_P(x)),d(q,\pi_P(y))\leq K_3$.
\end{lemma}

\begin{proof}
Let $\alpha'$ be the subpath of $\alpha$ from $x$ to $p$. Subdivide $\alpha'$ in subpaths $\alpha_1,\alpha_2$ so that $\alpha_1\cap B^{\hat{X}}_{K_2}(P)=\emptyset$ and $l(I)\leq LK_2+L^2+1$. We get the desired bound applying the first lemma to $\alpha_2$ and the second one to $\alpha_1$.
\end{proof}

This lemma readily implies the second part of the BCP property. Let us show the first part. Consider $L-$quasi-geodesic standard paths $\g_1,\g_2$ without backtracking with corresponding endpoints (in $X$ and) at $d_X-$distance at most 1 and suppose that $\g_1$ has a $P-$component $e_{p,q}$. Also, suppose that $\g_2$ does not have a $P-$component. We have to provide a bound for $d(p,q)$. Let $K=\max\{K_i\}$. We can split $\g_2$ into three subpaths $\alpha_1,\alpha_2,\alpha_3$ so that $\alpha_1,\alpha_3\cap B^{\hat{X}}_{K}(P)=\emptyset$ and $l(I)\leq 2LK+L^2+1$, where $\alpha_2:I\to \hat{X}$. As $\alpha_2$ does not contain a $P-$component, by the previous lemmas we easily get a bound on $d_X(p,q)$. For completeness, we remark that we just showed the following.

\begin{lemma}
 Let $(X,\calP)$ be relatively hyperbolic. Then for each $L$ there exists $K_4$ with the following property. Suppose that $d(\pi_P(x),\pi_P(y))\geq K_4$, for some $x,y\in X$ and $P\in\calP$. Then any $L-$quasi-geodesic standard path without backtracking with endpoints $x,y$ contains a $P-$component.
\end{lemma}

$\Leftarrow:$ For convenience, we will assume that each $P\in\calP$ is path-connected, which we can guarantee by adding suitable paths to $X$ in view of uniform coarse connectedness. Let $x,y\in X$ and consider a geodesic $[x,y]$ in $\hat{X}$ connecting them. Define $\eta(x,y)$ by substituting each $\calP-$components of $[x,y]$ with paths in the corresponding $P\in\calP$. Also, let $trans(x,y)$ be $[x,y]\backslash \bigcup \mathring{e_{p,q}}$, where the union is taken over all $\calP-$components of $[x,y]$. We want to show that Theorem \ref{guessgeod} applies.
\par
We note the following easy result.
\begin{lemma}
 For each $L$ there exists $K_0$ with the following property. If $\g$ is an $L-$quasi-geodesic standard path so that all its $\calP-$components have length bounded by $L$, then for all $x,y\in \g\cap X$ we have $d_{X}(x,y)\leq K_0 d_{\hat{X}}(x,y)+K_0$.
\end{lemma}
Condition $1)$ follows from the BCP property applied to a geodesic from $x$ to $y$ and a trivial path, together with the lemma. Condition $2)$ in the present setting follows from condition $3)$. In fact, it is enough to consider a polygon in $\hat{X}$ consisting of $[x',y']$, a suitable subpath of $[x,y]$, possibly trivial geodesics in $X$ of length at most $K$, where $K$ is so that all $P\in\calP$ are $K-$coarsely connected, and possibly trivial $\calP-$components.
\par
 Let us now show $3)$ (the containment is obvious in the $d_{\hat{X}}-$metric, not so in the $d_X-$metric). The proof will also show that given $L$-quasi-geodesic standard paths without backtracking $\g_0,\g_1$ with common endpoints we have a bound depending on $L$ only on $d_{Haus}^X(\g_0\cap X,\g_1\cap X)$. Consider a geodesic triangle $\g_0,\g_1,\g_2$. We can form an $L-$quasi-geodesic standard path without backtracking $\alpha$ with the same endpoints as $\g_0$ by concatenating subpaths $\g'_1,\g'_2$ of $\g_2$ and a suitable geodesic $\g$ so that $d(\g_0,\g)>1$, where $L$ depends on the hyperbolicity constant of $\hat{X}$ only. If $\g$ contains a sufficiently long $P-$component $e_{p,q}$, then $p,q$ are close to endpoints of a $P-$component of $\alpha$, and hence of either a $P-$component of $\g_1$ or a $P-$component of $\g_2$, as $\g_0$ and $\g$ do not have tied $\calP-$components.
\par
The proof of $3)$ is then readily completed given the following lemma.
\begin{lemma}
 For each $L$ there exists $K$ with following property. Let $\alpha_0, \alpha_1$ be $L-$quasi-geodesic standard paths without backtracking and with corresponding endpoints at distance at most $L$. Also, suppose that all $\calP-$components of $\alpha_0$ and $\alpha_1$ have length at most $L$.
Then for each $x\in \alpha_0\cap X$ we have $d_X(x,\alpha_0\cap X)\leq K$.
\end{lemma}

\begin{proof}
There exists $K_0$ depending on $L$ and the hyperbolicity constant of $\hat{X}$ with the following property. First, if $x,y\in \g_0$ then $d_{X}(x,y)\leq K_0 d_{\hat{X}}(x,y)+K_0$. Also, for $x$ as in the statement, we can find a geodesic $\g$ in $\hat{X}$ of length at most $K_0$ connecting $x'\in\g_0$ to $\g_1$, where $d_X(x',x)\leq K_0$, and so that the concatenation $\alpha$ of an initial subpath of $\alpha_0$, $\g$ and a final subpath of $\alpha_1$ is a $K_0-$quasi-geodesic standard path with no backtracking. By the BCP-property we see that there is a bound $K_1$ on the lengths of the $\calP-$components of $\alpha$, which can in turn be used, together with $l(\g)\leq K_0$, to give a bound of $d_X(x,\g_1\cap X)$.
\end{proof}

Condition $4)$ is clear. Property $(\alpha_1)$ can be shown as follows. If $x,y\in N_R(P_0)\cap N_R(P_1)$ for $P_0,P_1\in\calP$ then we can construct an $L-$quasi-geodesic standard path $\g_i$ without backtracking from $x$ to $y$, where $L$ depends on $R$ only, concatenating a geodesic in $X$ from $x$ to $P_i$, a $P_i-$component and a geodesic in $X$ from $P_i$ to $y$. The BCP property implies that such components have bounded length, which then easily gives $(\alpha_1)$.
\par
In the setting of $6)$, we can construct an $L-$quasi-geodesic standard path without backtracking $\g$ from $x$ to $y$, where $L$ depends on $k$ only, by concatenating a geodesic in $X$ from $x$ to $P$, a $P-$component and a geodesic in $X$ from $P$ to $y$. Recall that the proof of $3)$ gives that quasi-geodesic standard paths without backtracking $\g_0,\g_1$ with common endpoints have the property that $\g_0\cap X$ and $\g_1\cap X$ are at bounded Hausdorff distance. In order to conclude the proof of $6)$ we just need to use this fact with $\g_0$ a geodesic from $x$ to $y$ and $\g_1=\g$.
\end{proof}

\section{Weak $(*)-$ATG suffices}
In this section we show that in our main definition of relative hyperbolicity we can just require control on geodesics rather than $(1,C)-$almost-geodesics.

\begin{prop}[Characterization (RH0)]
\label{rh0}
Let $X$ be a geodesic metric space and let $\calP$ be a subsets of $X$. Then $(X,\calP)$ is relatively hyperbolic $\iff$ $\calP$ satisfies $(\alpha_1)$ and $(\alpha_2)$ and $X$ is weakly $(*)-$asymptotically tree-graded with respect to $\calP$.
\end{prop}

\begin{proof}
 It suffices to show that (RH0) implies (RH3).
We will use the following lemmas, that are proven in \cite{DS-treegr} and \cite{Si-proj} for relatively hyperbolic spaces.
In the following lemmas, $(X,\calP)$ is assumed to satisfy (RH3) and for $P\in\calP$ and $x\in X$ we denote by $\pi_P(x)$ a point so that $d(x,\pi_P(x))\leq d(x,P)+1$. Let $M$ be as in $(\alpha_2)$.

\begin{lemma}
\label{trh0}
There exists $t$ so that for each $L\geq 1$ any geodesic connecting points in $N_L(P)$, for some $P\in\calP$, is contained in $N_{tL}(P)$.
\end{lemma}

\begin{proof}
Fix $\epsilon$ and $L\geq 1$ as in $(\alpha_2)$. We can assume $L> M$. Let $\g$ be a geodesic as in the statement and consider a subgeodesic $[x,y]$ of $\g$ so that $\g'\cap N_L(P)=\{x,y\}$. We clearly cannot have $l(\g')\geq L/\epsilon$, which easily implies the lemma.
\end{proof}

\begin{lemma}\label{proj1:lem}
There exists $R$ such that each geodesic from $x$ to $N_M(P)$ intersects $N_R(\pi_P(x))$, for each $x\in X$ and $P\in\calP$.
\end{lemma}

\begin{proof}
For convenience, let us assume $M\geq \sigma$, for $\sigma$ as in $(P)$. Let $t$ be as in the previous lemma and $B$ be a uniform bound on the diameters of $N_{tM}(P_1)\cap N_{tM}(P_2)$ for $P_1,P_2\in\calP$, $P_1\neq P_2$. Set
$$R=\max\{tM+4\sigma+1, B+tM+3\delta+1, M+\delta+1\}.$$
Set $p=\pi_P(x)$ and consider some $y\in N_M(P)$ and a geodesic $\gamma$ from $x$ to $y$. Suppose that a geodesic triangle with vertices $x,y,p$ and such that $\gamma$ is one of its sides falls under case $(C)$.
Then we have points $a\in [x,p]$, $b\in [p,y]$ and $c\in\gamma$ with reciprocal distances at most $2\sigma$. We have that $d(a,p)\leq tM+2\sigma+1$, for otherwise
$$d(x,P)\leq d(x,a)+d(a,b)+d(b,P)< d(x,p)-tM-2\sigma+2\sigma+tM\leq d(x,P).$$
In particular $d(c,p)\leq d(c,a)+d(a,p)\leq tM+4\sigma+1$.
\par
Let us then consider a geodesic triangle with vertices $x,y,p$ such that $\gamma$ is one of its sides which falls under case $(P)$. Consider $A\in\calP$ and $x_1,x_2,y_1,y_2,p_1,p_2$ as in case $(P)$, where, e.g., $p_2,y_1\in[p,y]$ and $p_2$ is closer to $p$. Notice that $d(p,p_1)\leq tM+\delta+1$, for otherwise we would have
$$d(x,P)\leq d(x,p_1)+d(p_1,p_2)+d(p_2,P)< d(x,P).$$
If $d(p_2,y_1)\leq B$, we have
$$d(y_2,p)\leq d(y_2,y_1)+d(y_1,p_2)+d(p_2,p_1)+d(p_1,p)\leq $$
$$\delta+B+\delta+tM+\delta+1=B+tM+3\delta+1\leq R.$$

If $d(p_2,y_1)>B$ we have $A=P$. In this case, $d(x_1,p)\leq M+1$, because
$$d(x,P)\leq d(x,x_1)+d(x_1,P)\leq d(x,p)-d(x_1,p)+M\leq d(x,P)+1+M-d(x_1,p).$$

Therefore $d(x_2,p)\leq d(x_2,x_1)+d(x_1,p)\leq M+1+\delta\leq R$.
\end{proof}


Notice that $R\geq M$.

\begin{lemma}\label{proj2:lem}
There exists $L$ such that for each $x,y\in X$, $P\in\calP$, if $d(\pi_P(x),\pi_P(y))\geq L$, then any geodesic from $x$ to $y$ intersects $B_L(\pi_P(x))$ and $B_L(\pi_P(x))$.
\end{lemma}

\begin{proof}
Once again, let $B$ be a uniform bound on the diameters of $N_{tM}(P_1)\cap N_{tM}(P_2)$ for $P_1,P_2\in\calP$, $P_1\neq P_2$. Set
$$L=\max\{tR+4\sigma+1,2R+2\sigma+2, B+tR+3\delta+1, R, 2R+\delta+2\}.$$
Set $\hat{x}=\pi_P(x)$ and $\hat{y}=\pi_P(y)$. Consider a geodesic $\gamma$ from $x$ to $y$. Suppose that we have a case $(C)$ geodesic triangle with vertices $x,\hat{y},y$ containing $\gamma$. Therefore, we have points $a\in[x,\hat{y}]$, $b\in[\hat{y},y]$, $c\in\gamma$ with reciprocal distances at most $2\sigma$.
Use the previous lemma to find $p\in [x,\hat{y}]$ such that $d(\hat{x},p)\leq R$. If $a\in[p,\hat{y}]$, we have
$$d(y,P)\leq d(y,b)+d(b,a)+d(a,P)\leq d(y,P)-d(b,\hat{y})+2\sigma+tR+1,$$
therefore $d(b,\hat{y})\leq tR+2\sigma+1$. In this case $d(c,\hat{y})\leq tR+4\sigma+1\leq L$. On the other hand, if $a\in[x,p]$ then $d(\hat{y},a)=d(p,a)+ d(p,\hat{y})\geq d(p,a)+L-R$. Hence $d(\hat{y},b)\geq d(p,a)+L-R-2\sigma$ and
$$d(y,P)\leq d(y,\hat{y})-d(\hat{y},b)+d(b,a)+d(a,p)+d(p,\hat{x})\leq$$
$$ d(y,P)+1-d(a,p)-L+R+2\sigma+d(a,p)+R=d(y,P)+1+2R+2\sigma-L ,$$
which implies $L\leq 2R+2\sigma+1$, a contradiction.
\par
Suppose we have a case $(P)$ geodesic triangle as above. Consider $x_1,x_2,\hat{y}_1,\hat{y}_2,y_1,y_2$ and $A\in\calP$ as in case $(P)$ (where, e.g., $[x_1,\hat{y}_2]\in[x,\hat{y}]$ and $x_1$ is closer to $x$). Let $p$ be as above, and suppose first that $\hat{y}_2\in [p,\hat{y}]$. Proceeding as above we get $d(\hat{y}_1,\hat{y})\leq tR+\delta+1$ and hence $d(\hat{y},\hat{y}_2)\leq tR+2\delta+1$. If $d(\hat{y}_2,x_1)\leq B$, then
$$d(x_2,\hat{y})\leq d(x_2,x_1)+d(x_1,\hat{y}_2)+d(\hat{y}_2,\hat{y})\leq B+tR+3\delta+1.$$
Otherwise $A=P$. In this case, by Lemma \ref{proj1:lem}, there is a point on $z\in[y,y_2]\subseteq \gamma$ such that $d(z,\hat{y})\leq R\leq L$.
\par
We are only left to prove that the case $\hat{y}_2\in [x,p]$ is impossible. In fact, doing the estimates as above, we obtain $d(\hat{y},\hat{y}_1)\geq d(\hat{y}_2,p)+L-R-\delta$ and
$$d(y,P)\leq d(y,\hat{y})-d(\hat{y},\hat{y}_1)+d(\hat{y}_1,\hat{y}_2)+d(\hat{y}_2,p)+d(p,\hat{x})\leq d(y,P)+2R+\delta+1-L.$$
Hence $L\leq 2R+\delta+1$, a contradiction.
\end{proof}

 \begin{cor}\label{firstpoint:cor}
 For each $\mu\geq L$ there exists $R'$ with the following property. If $\gamma$ is a geodesic starting from $x$ and $p\in\gamma$ is the first point in $N_\mu(P)$ for some $P\in\calP$, then $d(p,\pi_P(x))\leq R'$.
 \end{cor}

 \begin{proof}
We have $d(p,\pi_P(p))\leq \mu+1$. Suppose $\mu>L$, and otherwise use the following argument to a point just before $p$ on $\gamma$. If we had $d(\pi_P(p),\pi_P(x))\geq L$ then we would have $[x,p]\cap N_L(P)\neq\emptyset$. As this is not the case, we have $d(\pi_P(p),\pi_P(x))< L$ and hence $d(p,\pi_P(x))\leq L+\mu+1$.
 \end{proof}

\begin{lemma}
\label{enhmrh0}
 Let $(X,\calP)$ be relatively hyperbolic. Then there exists $K$ so that the following holds. Suppose $x,y\in X$, $P\in\calP$ are so that $d(x,P)+d(y,P)\leq d(x,y)-K$. Then there exists $z\in[x,y]\cap N_{K}(P)$ with $|d(x,z)-d(x,P)|\leq K$.
\end{lemma}

\begin{proof}
It follows from the previous lemma that if $d(\pi_P(x),\pi_P(y))$ is large enough then any geodesic from $x$ to $y$ intersects balls of uniformly bounded radius around the projection points. Up to increasing $K$, we can make sure that the said condition holds.
\end{proof}

Property $2)$ follows directly from Lemma \ref{enhmrh0}. Let us now show $3)$, with $\mu=\max\{L,\sigma\}$ and $R_0$ large enough so that the following conditions are satisfied. First, we want that, for each geodesic $\g$, $deep_{\mu,R_0}$ is a disjoint union of subpaths each contained in $N_{t\mu}(P)$ for some $P\in\calP$. Also, we require that, for each geodesic $\g$ and $P\in\calP$, the entrance point of $\g$ in $N_{\mu}(P)$ is in $trans_{\mu,R_0}(\g)$ if $diam(\g\cap N_{\mu}(P))\geq 2R_0$. Both conditions can be arranged in view of $(\alpha_1)$ and Lemma \ref{trh0}. Fix $R\geq R_0$. An easy argument based on weak $(*)-$asymptotic tree-gradedness shows that it is enough to prove that for each $K$ there exists $D_1=D_1(K)$ so that whenever $[x,y],[x,y']$ are geodesics satisfying $d(y,y')\leq K$ then $trans_{\mu,R}([x,y])\subseteq N_{D_1}(trans_{\mu,R}([x,y']))$. It is easy to show that each $\pi_P$ is coarsely Lipschitz, see \cite[Lemma 2.4]{Si-contr}. Then, by Lemma \ref{proj2:lem} and Corollary \ref{firstpoint:cor}, we have that if $p\in [x,y]$ is the entrance point of $[x,y]$ in $N_{\mu}(P)$ (or the exit point) and $diam(N_{\mu}(P)\cap [x,y])\geq Q$, for some suitable $Q$, then $diam([x,y']\cap N_{\mu}(P))\geq 2R_0$ and the entrance point of $[x,y']$ in $N_{\mu}(P)$ is within bounded distance from $p$.
\par
Hence, we can further reduce $3)$ to the following. For each $K$ there exists $D_2=D_2(K)$ so that whenever $[x,y],[x',y']$ are geodesics satisfying $d(x,x'),d(y,y')\leq K$ and so that for each $P\in\calP$ we have $diam(N_{\mu}(P)\cap[x',y'])\leq Q$ then $trans_{\mu,R}([x,y])\subseteq N_{D_1}(trans_{\mu,R}([x',y']))$. Also, we can just show $[x,y]\subseteq N_{D_1}([x',y'])$, as with our hypothesis $trans_{\mu,R}([x',y'])$ is within bounded Hausdorff distance from $[x',y']$. Pick $p\in[x,y]$ and consider a geodesic triangle with vertices $x',y',p$. In view of our hypothesis, all sides of the said triangle intersect a ball of radius bounded by some $\sigma'$. Consider $x_1,y_1$ in, respectively, $[x',p]$ and $[p,y']$ with $d(x_1,y_1)\leq 2\sigma'$ and $d(y_1,[x',y'])\leq 2\sigma'$. Notice that
$$d(x',y')+4K \geq d(x',p)+d(p,y')=d(x',x_1)+d(x_1,p)+d(p,y_1)+d(y_1,y')\geq$$
$$ d(x',y')-4\sigma'+d(p,y_1),$$
so we get a bound on $d(p,y_1)$ in terms of $K$ and $\sigma'$, and hence a bound on $d(p,[x',y'])$, as required.
\end{proof}

The following fact will be used in \cite{S-compmap} and its proof is very similar to the previous one. Roughly speaking, the content is that quasi-geodesics with the same endpoints have the same transient set, up to finite Hausdorff distance.

\begin{prop}(Cfr. \cite[Proposition 8.14]{Hr-relqconv})
\label{transqgeod}
 Fix a relatively hyperbolic space $(X,\calP)$. For any $L,C$ there exist $\mu, R, M$ so that for each continuous $(L,C)-$quasi-geodesics $\delta,\gamma$ with the same endpoints we have
$$d_{Haus}(trans_{\mu,R}(\delta), trans_{\mu,R}(\gamma))\leq M.$$
Furthermore there exists $t\geq 1$ so that
\begin{enumerate}
 \item for every continuous $(L,C)-$quasi-geodesic $\gamma$, $deep_{\mu,R}(\delta)$ is a disjoint union of subpaths each contained in $N_{t\mu}(P)$ for some $P\in\calP$,
 \item for every $P\in\calP$ the entrance point of $\g$ in $N_{\mu}(P)$ is in $trans_{\mu,R}(\g)$ if $diam(\g\cap N_{\mu}(P))\geq 2R$.
\end{enumerate}
\end{prop}

\begin{proof}
 Similarly to the proof above, we need the following facts. Fix $L,C$.
\begin{lemma}\cite[Lemma 4.15]{DS-treegr}
\label{t}
 There exists $t$ so that for every $d\geq 1$ all $(L,C)-$quasi-geodesics with endpoints in $N_{d}(P)$, for some $P\in\calP$, are contained in $N_{td}(P)$. 
\end{lemma}

As usual, denote by $\pi_P$ a (coarse) closest point projection on $P\in\calP$.

\begin{lemma}\cite[Lemma 1.17]{Si-proj}
\label{proj2qgeod}
 There exists $\mu$ so that if $d(\pi_P(x),\pi_P(y))\geq \mu$ for some $x,y\in X$ and $P\in\calP$ then all $(L,C)-$quasi-geodesics from $x$ to $y$ intersect $B_{\mu}(\pi_P(x))$ and $B_{\mu}(\pi_P(x))$.
\end{lemma}

We also have the following, that can be obtained just as in Corollary \ref{firstpoint:cor}.

\begin{cor}
\label{firspointqgeod}
 There exists $A$ with the following property. Let $\delta$ be a continuous $(L,C)-$quasi-geodesic starting at $x$ so that $\{y\}=\delta\cap N_\mu(P)$ is the final point of $\delta$, for $\mu$ as in Lemma \ref{proj2qgeod} and some $P\in\calP$. Then $d(y,\pi_P(x))\leq A$.
\end{cor}

We are now ready for the proof. We will only show that $trans_{\mu,R}(\delta)$ is contained in a suitable neighborhood of $trans_{\mu,R}(\gamma)$, the other containment is symmetric. Let $t,\mu$ be as in the lemmas above. As usual, we use $(\alpha_1)$ and Lemma \ref{t} to choose $R$ large enough so that, for every continuous $(L,C)-$quasi-geodesic $\gamma$, $deep_{\mu,R}(\delta)$ is a disjoint union of subpaths each contained in $N_{t\mu}(P)$ for some $P\in\calP$. Also, we require that, for each continuous $(L,C)-$quasi-geodesic $\gamma$ and $P\in\calP$, the entrance point of $\g$ in $N_{\mu}(P)$ is in $trans_{\mu,R}(\g)$ if $diam(\g\cap N_{\mu}(P))\geq 2R$.
\par
By Lemma \ref{proj2qgeod} and Corollary \ref{firspointqgeod}, we have that if $p\in \gamma$ is the entrance point of $\gamma$ in $N_{\mu}(P)$ (the exit point behaves similarly) and $diam(N_{\mu}(P)\cap \gamma)\geq Q$, for some suitable $Q$, then $diam(\delta\cap N_{\mu}(P))\geq 2R$ and the entrance point of $\delta$ in $N_{\mu}(P)$ is within bounded distance from $p$.
\par
Hence, by passing to subpaths of $\gamma,\delta$ connecting exit and entrance points in neighborhoods of peripheral sets, we see that it suffices to show the following statement:
\par
For each $K$ there exists $D=D(K)$ so that whenever $\gamma,\delta$ are continuous $(L,C)-$quasi-geodesics connecting $x,y$ and $x',y'$ respectively, where $d(x,x'),d(y,y')\leq K$, and for each $P\in\calP$ we have $diam(N_{\mu}(P)\cap \gamma),diam(N_{\mu}(P)\cap \gamma)\leq Q$ then
$$trans_{\mu,R}(\delta)\subseteq N_{D}(trans_{\mu,R}(\gamma)).$$
\par
Also, we can just show $\delta\subseteq N_{D}(\gamma)$, as with our hypothesis $trans_{\mu,R}(\gamma)$ is within bounded Hausdorff distance from $\gamma$.
Pick $p\in \delta$. Let $\delta_1, \delta_2$ be the sub-quasi-geodesics of $\delta$ with endpoints $x',p$ and $y',p$ respectively. Consider the quasi-geodesic triangle with sides $\gamma, \delta_1, \delta_2$ (as we regard it as a quasi-geodesic triangle it does not matter that the endpoints do not coincide as they are within distance $K$). In view of our hypothesis and \cite[Lemma 8.17,Proposition 8.16-$(1)$]{DS-treegr} (which essentially give $(*)-$asymptotic-tree-gradedness for quasi-geodesic triangles), all sides of the said triangle intersect a ball of radius, say, $\sigma$. Consider $x_1,y_1$ in, respectively, $\delta_1$ and $\delta_2$ with $d(x_1,y_1)\leq 2\sigma$ and $d(y_1,\gamma)\leq 2\sigma$. As $x_1,y_1$ are on a quasi-geodesic on opposite sides of $p$ and their distance is bounded, we can also bound $d(y_1,p)$, and hence we get a bound on $d(p,\gamma)$.
\end{proof}

\section{Applications}
\subsection{The Bestvina-Bromberg-Fujiwara construction}

Let $\bf{Y}$ be a set and for each $Y\in\bf{Y}$ let $\calC(Y)$ be a geodesic metric space.
For each $Y$ let $\pi_Y:{\bf Y}\backslash\{Y\}\to\calP(\calC(Y))$ be a function (where $\calP(Y)$ is the collection of all subsets of $Y$). The authors of \cite{BBF} defines a function $d^{\pi}_Y$ so that
$$|d^{\pi}_Y(X,Z)-diam\{\pi_Y(X)\cup\pi_Y(Z)\}|\leq 2\xi,$$
whose exact definition we do not need.
Using the enumeration in \cite{BBF}, consider the following Axioms:
\begin{itemize}
	\item[(0)] $diam(\pi_Y(X))<+\infty$;
	\item[(3)] There exists $\xi$ so that $\min\{d^{\pi}_Y(X,Z),d^{\pi}_Z(X,Y)\}\leq \xi$;
	\item[(4)] There exists $\xi$ so that $\{Y:d^\pi_Y(X,Z)\geq \xi\}$ is a finite set for each $X,Z\in\bf{Y}$.
\end{itemize}

For a suitably chosen constant $K$, let $\calC(\{(\calC(Y),\pi_Y)\}_{Y\in\bf{Y}})$ be the path metric space consisting of the union of all $\calC(Y)$'s and edges of length 1 connecting all points in $\pi_X(Z)$ to all points in $\pi_Z(X)$ whenever $X,Z$ are connected by an edge in the complex $\calP_K(\{(\calC(Y),\pi_Y)\}_{Y\in\bf{Y}})$ whose vertex set is ${\bf Y}$ and $X,Z\in{\bf Y}$ are connected by an edge if and only if for each $Y\in{\bf Y}\backslash\{X,Z\}$ we have $d^{\pi}_Y(X,Z)\leq K$. An essential feature of this construction is the following observation from \cite{BBF}.
\begin{rem}
 When there is a group action on the union of the $\calC(Y)$'s compatible with the maps $\{\pi_Y\}$, there is an induced action on $\calC(\{(\calC(Y),\pi_Y)\}_{Y\in\bf{Y}})$.
\end{rem}

The following is proven in \cite{BBF}.

\begin{thm}[{\cite[Theorem 3.15]{BBF}}]\label{bbf}
 If $\{(\calC(Y),\pi_Y)\}_{Y\in\bf{Y}}$ satisfies Axioms $(0), (3)$ and $(4)$ and each $\calC(Y)$ is $\delta-$hyperbolic then $\calC(\{(\calC(Y),\pi_Y)\}_{Y\in\bf{Y}})$ is hyperbolic.
\end{thm}

We can improve this result to the following.

\begin{thm}
 If $\{(\calC(Y),\pi_Y)\}_{Y\in\bf{Y}}$ satisfies Axioms $(0), (3)$ and $(4)$ then $\calC(\{(\calC(Y),\pi_Y)\}_{Y\in\bf{Y}})$ is hyperbolic relative to $\{\calC(Y)\}$.
\end{thm}

This result probably follows from the proof of Theorem \ref{bbf}. Instead, we will use a trick to reduce this theorem to Theorem \ref{bbf}.

\begin{proof}
 It is readily checked that we can apply Theorem \ref{bbf} when substituting each $\calC(Y)$ with the union $\calH(\calC(Y))\cup Y$ of $Y$ and a combinatorial horoball along its defining net (and keeping the same $\pi_Y$). The space $\calC(\{(\calH(\calC(Y))\cup Y,\pi_Y)\})$ is a Bowditch space for $\calC(\{(\calC(Y)\cup Y,\pi_Y)\})$ if the nets chosen to construct $\calH(\calC(Y))$ are coarse enough depending on the constants $K,\xi$.
\end{proof}

As mentioned in the introduction, \cite{H-percomm} contains a stronger result than the theorem above. We included this result as the proof is very short and as it suffices for our next application.

\subsection{Hyperbolically embedded subgroups}

We refer the reader to \cite{DGO} for the definition of hyperbolically embedded subgroup. In the following theorem, the technical statement $4)$ is essentially just a more general set of hypotheses for \cite[Theorem 4.42]{DGO} that allows to carry out essentially the same proof, and statement $3)$ is just a simplified version of $4)$. The motivating setting for statement $4)$ is when $H_\lambda$ acts parabolically on a hyperbolic space and its orbit is, up to finite Hausdorff distance, a horosphere contained in the horoball $X^\lambda_1$, as in Corollary \ref{paract}.

\begin{thm}
\label{hypembchar}
 Let $\{H_\lambda\}_{\lambda\in\Lambda}$ be a finite collection of subgroups of the group $G$. Then the following are equivalent.
\begin{enumerate}
\item There exists a generating set $X$ so that $\{H_\lambda\}\hookrightarrow_h (G,X)$.
\item There exists a generating set $X$ (the same as in $(1)$) so that $\Gamma=Cay(G,X)$ is hyperbolic relative to the left cosets of the $H_\lambda$'s and $d_\Gamma|_{H_\lambda}$ is quasi-isometric to a word metric.
\item $G$ acts by isometries on a geodesic space $S$ that is hyperbolic relative to the orbits (of a given point) of the cosets of the $H_\lambda$'s, and each $H_\lambda$ acts properly.
\item $G$ acts by isometries on a geodesic space $S$ so that $S$ is hyperbolic relative to a $G$-invariant collection $\{X^\lambda_{g_{i,\lambda}}\}_{g_{i,\lambda}\in G/H_\lambda, \lambda\in\Lambda}$ and
\begin{enumerate}[(a)]
\item each $H_\lambda$ acts properly,
\item
 $g_{i,\lambda} H_\lambda s\subseteq X^\lambda_{g_{i,\lambda}}$ for some fixed $s\in S$,
\item $H_\lambda$ stabilises $X^\lambda_{1}$ and acts coboundedly on $N_1(S\backslash X^\lambda_{1}) \cap X^\lambda_{1}$.
\end{enumerate}
\end{enumerate}
\end{thm}

The implication $1)\Rightarrow 2)$ can probably also be proven as in the appendix of \cite{DS-treegr}.

\begin{proof}
 $1)\Rightarrow 2):$ We set $S=Cay(G,X)$. For $\calH=\bigcup H_\lambda\backslash\{1\}$, we have that $Cay(G,X\sqcup \calH)$ is a coned-off graph for $S$ with $\calP$ the collection of the left cosets of the $H_\lambda$'s. By definition of being hyperbolically embedded, $Cay(G,X\sqcup \calH)$ is hyperbolic.  We have to check the BCP property. It is easily seen that we can restrict to combinatorial paths instead of standard paths. We will use \cite[Proposition 4.14]{DGO}, which gives a quantity $D=D(L,n)$ with the following property. Consider an $n-$gon each of whose sides is either an $L-$quasi-geodesic (combinatorial path) in $Cay(G,X\sqcup \calH)$ or an isolated component (meaning not tied to any component of the other sides). Then each of the isolated components has length bounded by $D$. (The proposition actually gives a bound on the sum of such lengths but we will not need this.) Here we also used \cite[Lemma 4.11-$(b)$]{DGO} which shows the compatibility of the notion of length of components used in \cite[Proposition 4.14]{DGO} with the one we use.
\par
It is immediate to see that the first part of the BCP property holds. Also, the second part follows considering a quadrangle with two sides being subpaths of $\g_0,\g_1$, one side being a short geodesic connecting endpoints of $\g_0,\g_1$ and the last side being a $P-$component.
\par
$4)\Rightarrow 1):$ \cite[Theorem 4.42]{DGO} is a similar statement for actions on a hyperbolic space. We will adapt its proof to our setting.
\par
We now start with an application of a result from \cite{BBF}. Let ${\bf Y}$ be the collection of all $X^\lambda_{g_{i,\lambda}}$'s and denote by
$$\pi_Y(X)=\bigcup_{x\in X} \{y\in Y:d(x,y)\leq d(x,Y)+1\},$$
for $X,Y\in {\bf Y}$ with $X\neq Y$ (we identify $Y$ and $\calC(Y)$). As discussed in \cite[Lemma 4.3]{MS-prodtrees}, the axioms stated in the previous subsection are satisfied in view of Lemma \ref{proj2:lem} and the following fact from \cite{Si-proj} (which gives Axiom $(0)$ and will be used later).
\begin{lemma}
 \label{ap3}
Given a relatively hyperbolic space, there exists $C$ so that whenever $Q_1\neq Q_2$ are distinct peripheral sets we have
$$diam(\pi_{Q_1}(Q_2))\leq C.$$
\end{lemma}
By \cite[Theorem D]{BBF} we get that $\calP_K({\bf Y})=\calP_K(\{(\calC(Y),\pi_Y)\}_{Y\in\bf{Y}})$ is hyperbolic (it is actually a quasi-tree).
\par
We now wish to define a generating system $X$ of $G$ so that, first of all, the Cayley graph $Cay(G,X\sqcup \calH)$, where $\calH=\bigcup H_\lambda\backslash\{1\}$, is quasi-isometric to $\calP_K({\bf Y})$. Denote $s_\lambda= X^\lambda_{1}\supseteq H_\lambda s$. Recall that for each $h\in H_\lambda$ we have $hs_\lambda=s_\lambda$. For every edge $e$ in $\calP_K({\bf Y})$ going from some $s_\lambda$ to some $g s_\mu$ choose $x_e\in H_\lambda g H_\mu$ so that
$$d(s,x_e s)\leq \inf\{d(s,ys)| y\in H_\lambda g H_\mu\}+1.$$
In order to make the generating system symmetric, we can require $x_f=x_e^{-1}$ when $f$, in the notation above, is the edge $g^{-1}e$.
\begin{rem}(cfr. \cite[Remark 4.48]{DGO} )
 For $x_e$ as above there is an edge in $\calP_K({\bf Y})$ from $s_\lambda$ to $x_e s_\mu$. In fact, as $x_e=h_\lambda gh_\mu$ for some $h_\lambda\in H_\lambda, h_\mu\in H_\mu$, we have
$$d_{\calP_K({\bf Y})}(s_\lambda, x_e s_\mu)=d_{\calP_K({\bf Y})}(h_\lambda^{-1}s_\lambda, g h_\mu s_\mu) =d_{\calP_K({\bf Y})}(s_\lambda,gs_\mu)=1.$$
\end{rem}

The proof of \cite[Lemma 4.49]{DGO} applies verbatim to show the desired fact that $Cay(G,X\sqcup\calH)$ is quasi-isometric to $\calP_K({\bf Y})$ and hence hyperbolic.
\par
We now have to show our version of \cite[Lemma 4.50]{DGO}, which requires a slightly different proof.

\begin{lemma}
\label{450}
 There exists $\alpha$ so that if for some $Y\in {\bf Y}$ and $x\in X\cup \calH$ we have
$$diam(\pi_Y(\{s,xs\}))\geq \alpha$$
then $x\in H_\lambda$ and $Y=s_\lambda$ for some $\lambda\in \Lambda$.
\end{lemma}

\begin{proof}
Suppose first $x=x_e\in X$, for $e$ connecting $s_\lambda$ to $gs_\mu$ in $\calP_K({\bf Y})$. There are 3 cases to consider, and in each one we will get a contradiction for $\alpha$ large enough.
\par
Case 1: $s_\lambda \neq Y \neq x s_\mu s$. In this case
$$diam(\pi_Y(\{s,xs\}))\leq K+2\xi,$$
as $s_\lambda, g s_\mu$ are connected by an edge in $\calP_K({\bf Y})$.
\par
Case 2: $s_\lambda=Y\neq x s_\mu$. If $\alpha$ is large then any geodesic from $xs$ to $s$ passes $C-$close to $\pi_Y(xs)$, for some $C=C(S)$, by Lemma \ref{proj2:lem}. Notice that $\pi_Y(xs)\in Y\cap N_1(S\backslash Y)$ (as any geodesic from $xs$ to $\pi_Y(xs)$ intersects $Y$ within distance $1$ of $\pi_Y(xs)$). Hence, by $(c)$, there is $R$ so that we can pick $h\in H_\lambda$ with $d(hs,\pi_Y(xs))\leq R$, and $R$ only depends on the action.
\par
By the definition of $x$, we have $d(s,h^{-1}xs)\geq d(s,xs)-1$. On the other hand, for $q$ a point on a geodesic from $xs$ to $s$ within distance $C$ of $\pi_Y(xs)$, we have
\par\smallskip
\begin{center}
\begin{tabular}{r l}
 $d(s,h^{-1}xs)=$ & $d(hs,xs)\leq d(hs,q)+d(q,xs)=$ \\
 & $d(hs,q)+d(s,xs)-d(q,xs)\leq$ \\
 & $R+C+d(s,xs)-(\alpha-(R+C)),$\\
\end{tabular}
\end{center}
a contradiction for $\alpha>2(R+C)+1$.
\par
Case 3: $x s_\mu=Y\neq s_\lambda$. We can translate by $x^{-1}$ to reduce to the previous case (in view of our choice $x_{g^{-1}e}=x_e^{-1}$).
\par
Hence, in the hypothesis of the lemma, we must have $x\in H_\lambda$, for some $\lambda\in \Lambda$. Notice that
$\bigcup_{\lambda\in \Lambda} \pi_{s_\lambda}\left(\bigcup_{\mu\neq\lambda} s_\mu\right)$
has finite diameter by Lemma \ref{ap3} (and $s$ belongs to the set), so that for $\alpha$ large enough we can exclude that $s_\lambda\neq Y$. Therefore, $s_\lambda=Y$ as required.
\end{proof}
The final part of the proof of \cite[Theorem 4.42]{DGO} only uses \cite[Lemma 4.50]{DGO}, and hence it applies to our case as well in view of Lemma \ref{450}.
\end{proof}

\begin{cor}
\label{paract}
 Let $\{H_\lambda\}_{\lambda\in\Lambda}$ be a finite collection of subgroups of the group $G$. Then there exists a generating set $X$ so that $\{H_\lambda\}\hookrightarrow_h (G,X)$ $\iff$ $G$ acts on a geodesic space $S$ so that each $H_\lambda$ acts properly, is a maximal parabolic subgroup and there is an invariant system of disjoint horoballs $\calU$ in 1-1 correspondence with the cosets of the $H_\lambda$'s in $G$ so that each $H_\lambda$ acts coboundedly on the corresponding horosphere.
\end{cor}

\begin{proof}
The if part follows from the theorem since $S$ is hyperbolic relative to $\calU$. The only if part follows adding combinatorial horoballs to $S$ as in $2)$ in view of (RH1).
\end{proof}

The notion of weakly contracting element is defined in \cite{Si-contr} in terms of the existence of an action so that the closest point projection on an orbit satisfies essentially the same properties as a closest point projection on a peripheral set of a relatively hyperbolic space.

\begin{cor}
 \label{weakcontr}
Let $G$ be a group and $g\in G$ be an infinite order element. Then $g$ is weakly contracting if and only if it is contained in a virtually cyclic hyperbolically embedded subgroup.
\end{cor}

\begin{proof}
 The only if part is proven in \cite{Si-contr}. If $g$ is contained in the virtually cyclic subgroup $E(g)$ hyperbolically embedded in $(G,X)$ then the action of $G$ on $Cay(G,X)$ satisfies the requirements of the definition of weakly contracting in view of the properties of closest point projections on peripheral sets \cite{Si-proj}.
\end{proof}

We say that $H<G$ is a \emph{virtual retract} if there exists a subgroup $G'<G$ and a retraction $\phi_H:G'\to H'$, where $H'=G'\cap H$. For example, all subgroups of finitely generated abelian groups are virtual retracts.

\begin{cor}
\label{hypembsgp}
Suppose $\{H_1,\dots,H_{\lambda}\}$ is hyperbolically embedded in $G$, and let $K<G$. Suppose further that, for each $\lambda$, $H_\lambda\cap K$ is a virtual retract of $K$. Then $\{H_1\cap K,\dots,H_{\lambda}\cap K\}$ is hyperbolically embedded in $K$.
\end{cor}

It is quite possible that the corollary holds in greater generality.

\begin{quest}
 Can the virtual retraction condition be substituted by a more general condition? Is the corollary true without additional hypotheses?
\end{quest}

\begin{proof}
Let $\phi_{H_\lambda}:H'_\lambda\to K'_\lambda$ be as in the definition of virtual retract (with $K'_\lambda$ being the finite index subgroup of $H_\lambda\cap K$). Denote by $S$ the metric graph obtained adding edges $(gx,g\phi_{H_\lambda}(x))$ to $Cay(G,X)$, for each $x\in H'_\lambda$ and $g\in G$. Notice that $G$, and hence $K$, acts on $S$.
\par
We claim that $S$ is hyperbolic relative to the collection $\calP$ of the left cosets of the $H_\lambda$'s it contains. The easiest definition to check is (RH2). In fact, the natural coned-off graph $\hat{C}$ of $Cay(G,X)$ (the one obtained adding edges connecting each pair of vertices contained in a left coset of some $H_\lambda$) is a coned-off graph for $S$ as well, as $S$ is obtained adding to $Cay(G,X)$ some of the edges that have to be added to get $\hat{C}$. The BCP property for $S$ follows from the BCP property for $Cay(G,X)$ and the observation that $d_S\leq d_{Cay(G,X)}$.
\par
Notice that each $P\in\calP$ is within bounded Hausdorff distance in $S$ from a left coset of $H_\lambda\cap K$. Hence, in order to apply Theorem \ref{hypembchar}-$(3)$, we now only have to show that each $H_\lambda\cap K$ acts properly on $S$. We will actually show that the restriction of the metric of $S$ to $H_\lambda\cap K$ is quasi-isometric to a word metric. Denote by $\pi_\lambda: Cay(G,X)\to H_\lambda$ a closest point projection. Choose now any quasi-isometry $t: H_\lambda\to H'_\lambda$ (both endowed with word metrics) that restricts to the identity on $K'_\lambda$.
\par
We claim that $f_\lambda=\phi_{H_\lambda}\circ t \circ (\pi_\lambda|_G):(G,d_S|_G)\to K'_\lambda$ is Lipschitz (notice that the closest point projection is defined in terms of the metric of $Cay(G,X)$, and now we are regarding $G$ as endowed with the restriction of the metric of $S$). This will prove that $K'_\lambda$, and hence $H_\lambda\cap K_\lambda$, is quasi-isometrically embedded in $S$, as $f_\lambda$ restricts to the identity on $K'_\lambda$.
\par
It suffices to show that $f_\lambda$ maps the endpoints of any edge of $S$ to a set of uniformly bounded diameter. Let $e=(p,q)$ be an edge of $S$.
\par
Case 1: $e\subseteq Cay(G,X)$. Closest points projections in relatively hyperbolic spaces are coarsely Lipschitz \cite{Si-proj} (it is a consequence of Lemma \ref{proj2:lem}), so that $d_{Cay(G,X)}(\pi_\lambda(p),\pi_\lambda(q))$ can be uniformly bounded. But this implies that $d_{H_\lambda}(\pi_\lambda(p),\pi_\lambda(q))$ can be uniformly bounded as well as $H_\lambda$ is quasi-isometrically embedded in $Cay(G,X)$. The conclusion follows from the fact that $t$ and $\phi_{H_\lambda}$ are Lipschitz (when restricted to the vertices of their domain).
\par
Case 2: $e=(gx,g\phi_{H_\lambda}(x))$ and $g \in H_\lambda$. The edge is contained in $H_\lambda$, so that $\pi_\lambda$ acts as the identity, and we can conclude as $t$ and $\phi_{H_\lambda}$ are Lipschitz.
\par
Case 3: $e=(gx,g\phi_{H_\lambda}(x))$ and $g \notin H_\lambda$. The projection of a peripheral set on another peripheral set has uniformly bounded diameter by Lemma \ref{ap3}, so that $d_{Cay(G,X)}(\pi_\lambda(p),\pi_\lambda(q))$ can be bounded. We can conclude as in Case 1.
\end{proof}

\subsection{Divergence}
\label{div}
We now recall the definition of divergence of a metric space $X$, following \cite{DMS-div}. Choose constants $0<\delta<1$ and $\gamma\geq 0$.
For a triple of points $a, b, c \in X$ with $d(c, \{a, b\}) = r > 0$, let $div_{\gamma} (a, b, c; \delta)$ be the infimum of the lengths of paths connecting $a, b$ that avoid $B(c, \delta r - \gamma)$. (If no such path exists, set $div_\gamma (a, b, c; \delta) = \infty$.)
\begin{defn}
The divergence function $Div^X_\gamma (n, \delta)$ of the space $X$ is defined as
the supremum of all numbers $div_\gamma (a, b, c; \delta)$ with $d(a, b) \leq n$.
\end{defn}

For functions $f,g:\R^+\to\R^+$ write $f\preceq g$ if there exists $C$ so that $f(n)\leq Cg(Cn+C)+Cn+C$, and define $\succeq,\asymp$ similarly. By \cite[Corollary 3.12]{DMS-div}, there exist $\delta_0,\gamma_0$ so that when $X$ is a Cayley graph we have $Div^X_\gamma (n, \delta)\asymp Div^X_{\gamma_0} (n, \delta_0)$ whenever $0<\delta\leq \delta_0$ and $\gamma\geq \gamma_0$. Also, the $\asymp$-equivalence class of $Div^X_\gamma (n, \delta)$ is a quasi-isometry invariant (of Cayley graphs). We will write $Div^X(n)$ for (the $\asymp$-equivalence class of) $Div^X_{\gamma_0} (n, \delta_0)$.
\par
We now show the following.

\begin{thm}
 Suppose that the one-ended group $G$ is hyperbolic relative to the (possibly empty) collection of proper subgroups $H_1,\dots, H_k$. Then
$$e^n \preceq Div^G(n)\preceq \max\{e^n,e^nDiv^{H_i}(n)\}.$$
\end{thm}

(We write $\max\{e^n,e^nDiv^{H_i}(n)\}$ instead of $\max\{e^nDiv^{H_i}(n)\}$ because we allow the collection of subgroups to be empty.)

Not all super-exponential functions $f(n)$ satisfy $f(n)\asymp e^n f(n)$, hence the following questions arise.

\begin{quest}
 Does there exist a group $G$ of super-exponential divergence so that $e^nDiv^G(n)\not\asymp Div^G(n)$? If so, is it true that $Div^G(n)\preceq \max\{e^n,Div^{H_i}(n)\}$? Do there exist one-ended properly relatively hyperbolic groups with isomorphic peripheral subgroups but different divergence?
\end{quest}

In the case of finitely presented groups we have a cleaner statement, due to the following lemma.
\begin{lemma}
\label{atmostexp}
 The divergence of a finitely presented one-ended group is at most exponential.
\end{lemma}

\begin{proof}
 All pairs of points $x,y$, say with $r_x=d(x,1)\leq d(y,1)=r_y$, can be joined by a path in $\overline{B}(1,3r_y)\backslash B(1,r_x/3)$, see e.g. \cite[Lemma 5.6]{MS-qhyp}. The intersection of such subset and $G$ has $\preceq e^{r_y}$ points, and the conclusion easily follows.
\end{proof}

A relatively hyperbolic group is finitely presented if and only if its peripheral subgroups are, see \cite{DG-prespar} and \cite{Os-rh}. Hence we have the following.

\begin{cor}
 Suppose that the one-ended group $G$ is hyperbolic relative to its proper subgroups $H_1,\dots, H_k$. Then $Div^G(n)\asymp e^n$.
\end{cor}

We split the proof of the theorem in two propositions.

\begin{prop}
 Let the group $G$ be hyperbolic relative to its proper subgroups $H_1,\dots, H_k$. Then $Div^G(n)\succeq e^n$.
\end{prop}

\begin{proof}
 Denote $\calH=\{H_1,\dots,H_n\}$. By \cite{Os-hypel} there exists an infinite virtually cyclic subgroup $E(g)$ in $G$ so that $(G,\calH\cup\{E(g)\})$ is relatively hyperbolic. In particular, $E(g)$ is within bounded Hausdorff distance from a bi-infinite geodesic $\alpha$ so that, for each $\mu$, $diam(\alpha\cap N_\mu(gH_i))$ is uniformly bounded for each $g\in G$ and each $H_i$ (by $(\alpha_1)$ and quasi-convexity of peripheral sets). We would like to show that if the path $\beta$, which we assume to have length at least $1$ for simplicity, connects points $a,b\in \alpha$ on opposite sides of some $c\in\alpha$ then $d(c,\beta)\leq C\log_2(l(\beta))+C$ for some constant $C$. In order to do so, we modify a standard argument which can be found, e.g., in \cite[Proposition III.H.1.6]{BrHa-metspaces} (and which we already used in Lemma \ref{logest}).
For $\mu$ as in (RH3) and $R$ large enough we have $trans_{\mu,R}(\alpha)=\alpha$, and similarly for all subgeodesics of $\alpha$. Let $q\in\beta$ be so that $l(\beta|_{[a,q]})=l(\beta|_{[q,b]})$ and consider a geodesic triangle with vertices $a,b,q$. We have some $D$ so that $\alpha|_{[a,b]}=trans_{\mu,R}(\alpha|_{[a,b]})\subseteq N_{D}(trans_{\mu,R}([a,q])\cup trans_{\mu,R}([q,b]))$. We can assume by induction (starting with paths of length at most, say, $100$), that $trans_{\mu,R}([a,q])\subseteq N_{D\log_2(l(\beta)/2)+100}(\beta)$ and similarly for $[q,b]$, and we get $\alpha_{[a,b]}\subseteq N_{D\log_2(l(\beta))+100}(\beta)$ as required.
\end{proof}

The author was not able to find a reference for the following fact in the case of hyperbolic groups.

\begin{prop}
 Let the one-ended group $G$ be hyperbolic relative to $H_1,\dots, H_k$. Then $Div^G(n)\preceq \max\{e^n,e^nDiv^{H_i}(n)\}$.
\end{prop}

\begin{proof}
 It is enough to find a bound on $div_\gamma (a, b, c; \delta)$ when $a,b,c$ are elements of $G$ (as opposed to points in the interior of edges). Fix $a,b,c\in G$, and suppose without loss of generality $r=d(a,c)\leq d(b,c)$. We would like to find a ``short'' path connecting $a$ to $b$ outside $B=B(c,r/\delta-\gamma)$, for $\delta>0$ small enough and $\gamma$ large enough. We can assume $r\leq 10n$ for $n=d(a,b)$ (for otherwise a geodesic from $a$ to $b$ would avoid $B$). As $G$ is one-ended, by \cite[Lemma 3.4]{DMS-div} there exists a path $\beta$ connecting $a,b$ outside $B$. Consider $a=x_0,\dots,x_k=b$ on $\beta$ so that $d(x_i,x_{i+1})\leq 1$, where $x_i\in G$. Define $y_i\in G$ to be the first point of $[x_i,c]$ in $B'=B(c,d(b,c))$. We then have a sequence of points $a=y_0,\dots,y_k=b$ in $(B'\cap G)\backslash B$, which contains $\preceq e^n$ elements.
\par
If $G$ was hyperbolic, we would have a uniform bound on $d(y_i,y_{i+1})$. Then it would be readily seen that we can choose $0=j_0<\dots<j_l=k$ so that $d(y_{j_i},y_{j_{i+1}})$ is bounded and $l\preceq e^n$, so that the desired path can be obtained concatenating geodesics of the form $[y_{j_i},y_{j_{i+1}}]$.
\par
In the relatively hyperbolic case, it is easily seen by the relative Rips condition (and the deep components being contained in neighborhoods of peripheral sets) that for an appropriate $D$ we have either $d(y_i,y_{i+1})\leq D$ or there exists a left coset $P_i$ of some $H_i$ so that $y_i,y_{i+1}\in N_D(P_i)$. Also, all geodesics from $c$ to $P_i$ pass $D-$close to some $z_i\in P_i$, see Lemma \ref{proj1:lem}. Assume for convenience that we are considering a Cayley graph with respect to a generating set containing generating sets for the $H_i$'s, so that $P_i$ is connected. As $d(y_i,y_{i+1})\leq 2d(c,b)$ and peripheral subgroups are undistorted, there exists a path $\alpha$ in $P_i$ of length $\preceq Div^{P_i}(2d(c,b))$ connecting points $y'_i,y'_{i+1}$ at distance at most $D$ from $y_i,y_{i+1}$ which avoids a $P_i-$ball of radius $\delta_0\min\{d(y'_i,z_i),d(y'_{i+1},z_i)\}-\gamma_0$ around $z_i$. Again as peripheral subgroups are undistorted, we have that $\alpha$ avoids $B$, if we chose $\delta,\gamma$ large enough. Now, we can choose $0=j_0<\dots<j_l=k$ so that $l\preceq e^n$ and $y_{j_i}$ can be connected to $y_{j_{i+1}}$ by a path avoiding $B$ and of length $\preceq \max\{Div^{H_i}(n)\}$. So, we can construct a path of length $\preceq \max\{e^n Div^{H_i}(n)\}$ connecting $a$ to $b$.
\end{proof}

\subsection{Approximation with tree-graded spaces}
\label{tg}
A standard result for hyperbolic spaces, see e.g. \cite[Lemma 2.12]{Ghys-dlH-90-hyp-groups}, is that any finite
configuration of points and geodesics can be approximated by some tree, the error in the approximation depending on the hyperbolicity constant and the number of points and geodesics involved only. This result is very useful to reduce computations in hyperbolic spaces to the tree case. We generalize this to relatively hyperbolic spaces. Fix from now on a relatively hyperbolic space $(X,\calP)$ and $\mu,R$ as in (RH3).
\begin{defn}
Let $\calA$ be a collection of points and peripheral sets. We denote by $\gamma(\calA)$ the union of $\calA$ and all $trans_{\mu,R}(\gamma)$ for $\gamma$ a geodesic connecting points each lying on some $A\in\calA$.
\end{defn}

We will approximate finite configurations in $X$ with tree-graded spaces, a notion that has been defined in \cite{DS-treegr}.

\begin{defn}
A geodesic complete metric space $\F$ is \emph{tree-graded} with respect to a collection of closed geodesic subsets of $\F$ (called \emph{pieces}) if the following properties are satisfied:
\par
$(T_1)$ distinct pieces intersect in at most one point,
\par
$(T_2)$ each geodesic simple triangle is contained in one piece.
\end{defn}

For the purposes of this section, a tree-graded space can be thought of as a tree-like arrangement of pieces. An example of tree-graded space is a the Cayley graph of $A*B$ with respect to a generating set which is th union of generating sets of $A$ and $B$.

\begin{lemma}
 Let $(X,\calP)$ be relatively hyperbolic and fix an approximating graph $\Gamma_P$ for each $P\in\calP$ as well as $\mu,R$ as in Proposition \ref{rh1}. For each $n\in\N$ there exists $C$ with the following property. Suppose that $\calA$ is a collection of points and peripheral sets and $\calA$ has cardinality at most $n$. Then there exists a tree-graded space $T$, each of whose pieces is isometric to some $\Gamma_P$, and a $(C,C)-$quasi-isometric embedding $f:\gamma(\calA)\to T$.
\end{lemma}

The idea is just to apply the statement for hyperbolic spaces in the Bowditch space and to use Proposition \ref{rh1}, ``moreover'' part included, to translate the information we get back to $X$.

\begin{proof}
 Fix the Bowditch space $Bow(X)$ with respect to the approximating graphs $\{\Gamma_P\}$ and let $\calA$ be as in the statement. Let $\hat{\calA}$ be the collection of all points in $\hat{\calA}$ and quasi-geodesic lines (with uniform constants) which are asymptotic to vertical rays in a horoballs corresponding to some $P\in\calA\cap \calP$ in both directions.

Denoting by $\gamma(\hat{A})$ the collection of all geodesics with endpoint in $\hat{A}$, we have that there exists a $(1,\hat{C})-$quasi-isometry $\hat{g}:\hat{T}\to \gamma(\hat{\calA})$ with quasi-inverse $\hat{f}$, for some tree $\hat{T}$ and $\hat{C}=\hat{C}(n)$. We will denote the convex core of a set $S$ in $\hat{T}$ again by $\gamma(S)$.

 By substituting each combinatorial horoball $\calH(\Gamma_P)$ with (the full subgraph with vertex set) $\Gamma_P\times (k+\N)$ for a suitable $k=k(\hat{C},Bow(X))$, we can make sure that $S_P=\gamma(\hat{f}(\calH(\Gamma_P)))$ are all disjoint and the distance between two such sets is at least some $\epsilon>0$.
\par
For each $P\in\calP$, remove $\mathring{S_P}$ from $\hat{T}$ and construct $T$ by gluing copies of $\Gamma_P$ so that each each point $p\in\partial S_P$ is glued to a point $q\in\Gamma^0_P$ at minimal distance from $\hat{g}(p)$. Notice that $T$ is tree-graded.
\par
We have that whenever a geodesic $\alpha$ in $X$ and a geodesic $\beta$ in $Bow(X)$ have the same endpoints then $trans_{\mu,R}(\alpha)$ is within Hausdorff distance bounded by some $D=D(k, Bow(X))$ from $\beta\cap X$. Also, moving the endpoint of a geodesic inside a peripheral set affects the transient set only up to finite Hausdorff distance. Using these two facts one can show that each $x\in\gamma(\calA)$ lies within uniformly bounded distance from some $y=y(x)\in \gamma(\hat{\calA})\cup \bigcup_{P\in\calP\cap\calA} \Gamma_P$. Define $f(x)=\hat{f}(y)$ if $y$ lies in $\gamma(\hat{\calA})$ and $f(x)=y$ if $y\in \bigcup_{P\in\calP\cap\calA} \Gamma_P$ (notice that such $\Gamma_P$'s are actually subsets of $T$).
\par
Let $x,x'\in \gamma(\calA)$ and consider $y=y(x),y'=y(x')$ as above. Let $\hat{\gamma}$ be a geodesic from $y$ to $y'$ and notice that $\hat{\gamma}\cap X$ is contained in a suitable neighborhood of $\gamma(\hat{\calA})\cup\{y,y'\}$. As we shrinked the horoballs, we can subdivide $\hat{\gamma}$ into subgeodesics either of length at least $k>0$ and contained in $N_{k}(X)$ (type 1) or contained in a horoball (type 2). We know (similarly to Corollary \ref{rh1geod}) that we can obtain a quasi-geodesic in $X$ from $y$ to $y'$ by substituting the subgeodesics $\gamma_i$ of type 1 with geodesics in $X$ whose endpoints are within distance $k$ from the endpoints of the $\gamma_i$'s and the geodesics of type 2 with geodesics each contained in a suitable neighborhood of a peripheral set. Also, each subgeodesic of type 1 has endpoints close to the image through $\hat{g}$ of some geodesic $\alpha_i$ in $T$ (of length at least $\epsilon>0$), and the $\alpha_i$'s can be concatenated with geodesics each contained in some $\Gamma_P$ to get a geodesic from a point close to $f(x)$ to a point close to $f(y)$.

Let us write $A\approx B$ if the quantities $A,B$ differ by some multiplicative and additive constants depending on the data in the statement of the lemma. From the discussion above we have, denoting $\calT_i$ the collection of subgeodesics as above of type $i$,
$$d(x,x')\approx \sum_{\gamma_i\in\calT_1} l(\gamma_i)+\sum_{[x_i,y_i]\in\calT_2} d_X(x_i,y_i)\approx d(f(x),f(x')),$$
hence $f$ is a quasi-isometric embedding.
\end{proof}

\subsection{Combination theorems}
\label{comb}
Theorem \ref{guessgeod} can be used to give alternative proofs of some known combination theorems for relatively hyperbolic groups from \cite{Da-comb}. As the results are known, we consider only one example and provide a sketch of proof only.
\begin{thm}[\cite{Da-comb}]
 Let $G_1, G_2$ be hyperbolic relative to a common subgroup $H$. Then $G=G_1 *_H G_2$ is hyperbolic relative to $H$.
\end{thm}

\emph{Sketch of proof.} Let $x,y\in G$. All paths from $x$ to $y$ have to cross a sequence $x\in X_1,\dots, X_n\ni y$ of left cosets of $G_1$ and $G_2$ corresponding to a geodesic in the Bass-Serre tree connecting vertices corresponding to $x$ and $y$. Define $\eta(x,y)$ by concatenating in the suitable order geodesics from $x$ to $X_1\cap X_2$, from $X_{n-1}\cap X_n$ to $y$ and from $X_{i-1}\cap X_{i}$ to $X_{i}\cap X_{i+1}$, and let $trans(x,y)$ be the set of all transient points for the said geodesics (with suitable constants). Let us for example check $3)$. Fix $x,y,z\in G$. Consider the corresponding tripod in the Bass-Serre tree we see that each point $p\in trans(x,y)$ lies in some geodesic $\gamma$ contained in a coset $X$ of $G_1$ or $G_2$ satisfying the following property. There exists cosets $Y_1,Y_2$ and possibly of $Y_3$ of $H$ contained in $X$ so that $\gamma$ connects $Y_1$ to $Y_2$ and either
\begin{itemize}
 \item $\eta(x,z)$ or $\eta(z,y)$ contains a subgeodesic with endpoints in $Y_1, Y_2$, or
 \item $\eta(x,z)$ contains a geodesic connecting $Y_1$ to $Y_3$ and $\eta(y,z)$ contains a geodesic connecting $Y_3$ to $Y_2$.
\end{itemize}

In both cases, by the relative Rips condition for $m-$gons with $m\leq 6$, we get the desired containment.\qed

\bibliographystyle{alpha}
\bibliography{metricrelhyp}
\end{document}